\documentclass[9pt,english]{article}
\usepackage{geometry}
\usepackage{amsmath,amsfonts, amsthm,amssymb,oldgerm,amssymb}
\usepackage{array}
\usepackage[english]{babel}
\usepackage{float}
\usepackage{url}
\usepackage{cite} 
\usepackage{hyperref}
\usepackage{xcolor} 
\usepackage{amssymb,mathrsfs,graphicx,amsmath,graphicx,amsthm}
\numberwithin{equation}{section}
\usepackage{mathtools}
\usepackage{array,multicol,comment,lipsum,caption}
\usepackage{color}
\usepackage{soul} 

\theoremstyle{theorem}
\newtheorem{theorem}{Theorem}[section]

\theoremstyle{definition}
\newtheorem{definition}{Definition}[section]

\theoremstyle{lemma}
\newtheorem{lemma}{Lemma}[section]

\theoremstyle{remark}
\newtheorem{remark}{Remark}[section]

\theoremstyle{proposition}

\newtheorem{corollary}[theorem]{Corollary}

\usepackage{epsfig}
\usepackage{relsize}
\usepackage{hyperref}
\usepackage[utf8]{inputenc}

\newcommand{\Sum}{\displaystyle\sum}

\newcommand{\R}{{\mathbb{R}}}

\DeclarePairedDelimiter\abs{\lvert}{\rvert}

\DeclarePairedDelimiterX{\inp}[2]{\langle}{\rangle}{#1, #2}

\newcommand{\mycomment}[1]{}

\makeatletter
\newcommand{\myitem}[1]{%
\item[#1]\protected@edef\@currentlabel{#1}%
}

\begin{document}
\title{\bf An Inverse Problem for the Moore–Gibson–Thompson Equation: Injectivity and Stability
}

\author{Benjam\'in Palacios\ \footnotemark[2] \hspace{.1cm }and Alexander Rodr\'iguez\footnotemark[3]}


\footnotetext[2]{Facultad de Matem\'aticas, Pontificia Universidad Católica de Chile, Avda.
Vicu\~na Mackenna 4860, Macul, Santiago, Chile. E.mail: benjamin.palacios@uc.cl}

\footnotetext[3]{Facultad de Matem\'aticas, Pontificia Universidad Católica de Chile, Avda.
Vicu\~na Mackenna 4860, Macul, Santiago, Chile. E.mail: jarodriguez19@uc.cl}

\date{\today}

\maketitle

\begin{abstract}
We study an inverse problem for the variable–coefficient Moore–Gibson–Thompson (MGT) equation in a bounded domain, where the objective is to recover an unknown initial acceleration from boundary measurements of the acoustic field. We establish both injectivity and stability of the associated measurement operator. Injectivity is established through a unique continuation property by applying a theorem of Tataru for analytically principally normal operators. The argument relies on a factorization of the third–order operator into wave-type components, combined with Carleman estimates and propagation techniques in the spirit of Tataru. Stability is derived through a microlocal analysis of the forward operator, based on the construction of geometric optics approximate solutions and the identification of an elliptic pseudodifferential operator arising from microlocal back–propagation of the data. These results connect the MGT inverse problem with the classical theory of hyperbolic inverse problems, such as photoacoustic tomography, while highlighting new features induced by the third–order dynamics and damping.
\end{abstract}

\section{Introduction}
The mathematical modeling of nonlinear acoustic wave propagation in viscous and thermally relaxing fluids has been the subject of sustained research for several decades, driven by applications in medical imaging, therapeutic ultrasound, industrial ultrasonics, and sonochemistry. Classical second-order models such as Kuznetsov’s equation, the Westervelt equation, and the KZK equation provide accurate descriptions in many regimes and have been extensively studied in both the physical and mathematical literature; see, for instance, \cite{kuznetsov1971equation,westervelt1963parametric,hamilton2024nonlinear} and the surveys by \cite{crighton1979model,coulouvrat1992equations}.\\

Let $\Omega \subset \mathbb{R}^n$ $(n \ge 2)$ be a bounded domain with smooth boundary $\partial\Omega$ and let 
$Q=\Omega\times(0,T)$ for some $T>0$. 
We consider the following third-order Moore--Gibson--Thompson (MGT) equation
\begin{equation}
\begin{cases}
\tau(x) u_{ttt} + \alpha(x) u_{tt} 
- c^{2}(x)\,\Delta u
- b(x) \Delta u_t=0,
& (x,t)\in Q,\\[1mm]
u(x,0)=0,\quad u_t(x,0)=0,\quad u_{tt}(x,0)=f(x),
& x\in\Omega,
\end{cases}
\tag{1.1}
\end{equation}
where the coefficients $\tau(x), b(x),$ 
$c(x)$ and $\alpha(x)$ are smooth real-valued functions. 
This PDE involves coefficients that depend on the physical properties of the medium and must therefore be treated as spatially varying functions. In applications such as ultrasound propagation in heterogeneous biological tissues \cite{greenleaf1974algebraic,dines1979ultrasonic,gemmeke2016,alaskar2021,javaherian2020}, the sound speed $c=c(x)$ is a key imaged quantity, while the sound diffusivity $b=b(x)$ accounts for attenuation effects and carries tissue-specific information. Additionally, the coefficient $\alpha=\alpha(x)$, multiplying the second-order time derivative, is linked to thermoviscous relaxation, attenuation, and material nonlinearity \cite{bjornum1986,cain1986,ichida1983,burov1994}. Allowing these coefficients to depend on the spatial variable provides a means of imaging material properties of the medium, for instance in ultrasound tomography.\\

Inverse problems associated with the Moore–Gibson–Thompson equation and its nonlinear version, the Jordan–Moore–Gibson–Thompson equation, have received increasing attention in recent years. Early contributions focused on the identification of time-dependent or space-dependent coefficients from boundary measurements, motivated by applications in ultrasound imaging. In particular, on \cite{liu2013} they studied the recovery of a time-dependent coefficient in a third-order acoustic equation by means of a single boundary observation. More recently, on \cite{kaltenbacher2021,kaltenbacher2022} they investigated several inverse problems for nonlinear acoustic models, including the identification of nonlinear damping parameters and space-dependent coefficients from boundary data.\\

Stability results for partial data inverse problems involving time-dependent viscosity and potential terms in the MGT equation were established in \cite{fuyaoyu2026} using Carleman estimates and geometric optics solutions. In addition, inverse problems for related third-order models have been addressed in the context of fractional and relaxation-type equations, where the recovery of coefficients or parameters was studied under various observation settings, see for instance \cite{kaltenbacher2022nikolic}. These works demonstrate that inverse problems for MGT-type equations are closely tied to unique continuation properties and require refined analytical tools adapted to the third-order time structure of the model.\\

Following the ideas given in \cite{UHLMANN}, we establish a unique continuation property for the MGT operator via Carleman estimates, extending continuation principles for hyperbolic equations with memory to a higher-order setting. Although the result requires conditions on the whole boundary rather than partial data, the main challenge stems from the lack of ellipticity caused by the third-order time structure and space–time coupling. To overcome this, we use a Tataru’s theorem based on analytic principal normality, verify that the MGT principal symbol satisfies this property, and thereby obtain unique continuation despite the non-elliptic framework. Motivated by the geometric microlocal approach given in \cite{StefanovUhlmann2011}, we further construct a parametrix and derive a stability estimate for the associated inverse problem.

\section{Mathematical Setting and Inverse Problem}

Let's define the variable coefficient Moore-Gibson-Thompson (MGT) operator
\begin{equation}\label{eq:MGT-operator}
Pu := \tau(x)\, u_{ttt} + \alpha(x)\, u_{tt} - c^2(x)\, \Delta u - b(x)\, \Delta u_t,
\end{equation}
with $c, \tau, b$ and $\alpha$ all smooth, bounded, real-valued functions, such that
\begin{equation}\label{hyp:lower_bounds}
\tau\geq\tau_0,\quad \alpha\geq \alpha_0,\quad c\geq c_0\quad b\geq b_0,
\end{equation}
for fixed constants $\tau_0,\alpha_0,c_0,\beta_0>0$, and consider the initial value problem
\begin{equation}\label{eq:forward}
\begin{cases}
Pu = 0, & \text{in } \R^n\times(0,T) , \\[6pt]
(u,\; u_t,\; u_{tt})|_{t=0} = \mathbf{f}, & \text{in } \R^n,
\end{cases}
\end{equation}
where $\mathbf{f}=(f_1,f_2,f_3)$ is a given initial condition supported inside the bounded region $\Omega$.

Given that the MGT equation originates from the Maxwell–Cattaneo law for heat conduction (see \cite{kaltenbacher2011wellposedness}), solutions to this equation possess the finite-speed-of-propagation property. Well-posedness for this system then follows from \cite{nicaise2020wellposedness} by considering null Dirichlet conditions on the boundary of a sufficiently large domain $\Omega' \supset \Omega$ (which depends on $T$). Indeed, for
\[
(f_1,f_2,f_3) \in H_0^1(\Omega') \times L^2(\Omega') \times L^2(\Omega'),
\]
there exists a unique solution $u$ satisfying
\[
u \in C([0,T]; H_0^1(\Omega')), \quad
u_t \in C([0,T]; L^2(\Omega')), \quad
u_{tt} \in C([0,T]; L^2(\Omega')),
\]
with continuous dependence on the initial data.\\

This paper concerns the study of the inverse problem that seeks to determine the initial source of waves from measurements acquired at the boundary of the region of interest $\Omega$, and for all time in the interval $(0,T)$. Let $\Gamma\subset\partial\Omega$ be an open subset of the boundary, which will represent the observation region, and let's denote
\[
Q := \Omega\times(0,T) \qquad \text{and}\qquad
\Sigma := \Gamma\times(0,T) .
\]
When measurements are available on the whole boundary we set $ \Gamma := \partial\Omega$, and we will refer to this as the complete data (inverse) problem. The observation map is the operator that assigns to each initial source $\mathbf{f}$, the Dirichlet boundary measurement
\begin{equation}\label{eq:observation}
\Lambda \mathbf{f} := u\big|_{\Sigma} .
\end{equation}

In the particular case of $\mathbf{f}= (0,0,f_3)$, the inverse problem consists in determining the initial acceleration $f_3$ from the knowledge of $\Lambda \mathbf{f}$. In this scenario, we will simply write $\Lambda f_3$.\\

We point out that the inverse problem considered here is closely related, both in formulation and in analytical structure, to the classical problem of photoacoustic tomography (PAT) for the wave equation. In PAT, one seeks to reconstruct an initial pressure distribution from acoustic boundary measurements governed by the wave equation, and this problem has been extensively studied in the mathematical literature with respect to uniqueness, stability, and reconstruction formulas; see, for example, the seminal works of Stefanov and Uhlmann on thermoacoustic and photoacoustic inversion with variable sound speed and partial data \cite{MR2519863,StefanovUhlmann2011}. These studies establish fundamental properties of the measurement operator in the undamped wave setting and demonstrate how boundary measurements can be used to obtain a unique determination of the initial condition, while stability estimates under partial data are obtained under suitable geometric conditions. The PAT inverse problem has also been extensively studied in the setting of attenuating media and for different types of attenuation models. See, for instance, \cite{Homan2013,BENJAMIN-ACOSTA-THERMO,Palacios2022,huang2023determination,acosta2025inverse}. Our work, then, can be viewed as an extension of this rich framework to the Moore–Gibson–Thompson setting, where both the higher-order dynamics and damping introduce additional analytical challenges. 

Recently, the microlocal analysis of waves with memory-type attenuation and its application to coefficient inverse problems has received considerable attention. Works along these lines include \cite{covi2026propagation,kian2026inverse,grebnev2026high,yu2026uniqueness}.

The main result of this paper establishes injectivity and a stability estimate for the measurement operator $\Lambda$. We will state it informally as follows.
\begin{theorem}\label{main-thm}
Let $T>0$ and let $\Omega \subset \mathbb{R}^n$ be an open and bounded domain with a smooth and strictly convex boundary for the metric $\frac{\tau}{b}dx^2$., and let $\mathcal{K}\subset\Omega$ compact subset. Assume that the domain $\Omega$ can be foliated by a particular family of smooth and strictly convex hypersurfaces, and
the supremum of the length of geodesic segments inside $\Omega$ and passing through $\mathcal{K}$ is less than $2T$.
Then, the observation map $\Lambda$ for the complete data problem (defined in \eqref{eq:observation}) is injective; that is,  
if $f_3 \in L^2(\Omega)$ and $\Lambda f_3 = 0$, then $f_3 \equiv 0$ in $\Omega$. Moreover, there is a constant $C_{\mathcal{K}}>0$ such that the following stability estimate holds:
\[
\| f_3 \|_{L^2(\mathcal{K})} \le C_{\mathcal{K}}\, \| \Lambda f_3 \|_{H^1(\partial\Omega\times (0,T))},
\qquad \forall f_3 \in L^2(\mathcal{K}).
\]
\end{theorem}

\begin{remark}
The previous theorem provides a unified statement of the main result. The detailed assumptions and formulations are given in Corollary \ref{cor:uniqueness} and Corollary \ref{cor:stability-thm}, which address separately the injectivity of the observation map and the corresponding stability estimate respectively.
\end{remark}
\medskip

We emphasize that the requirement for global boundary measurements is dictated by the precise foliation considered in the proof of injectivity (see Theorem \ref{Theorem-unique-continuation}). 
We note, however, that partial boundary measurements may also be considered under additional geometric and technical assumptions; see, for instance, \cite{acosta2025inverse} for a related result. On the other hand, the stability analysis of the inverse problem is carried out under partial boundary observations, showing that the stable reconstruction of the unknown initial datum $f_3$ remains possible even when the data are available only on a suitable subset of the boundary.\\

The rest of the paper is organized as follows. 
In section \ref{sec:UC} we establish the injectivity of the measurement operator by proving a unique continuation property for the Moore--Gibson--Thompson equation. 
The proof follows an argument introduced in \cite{UHLMANN} and later generalized in \cite{stefanov2013recovery,acosta2025inverse}. It consists of two main steps which are then iterated a finite number of times. In the first step, a Carleman estimates for small times allows to show unique continuation in a neighborhood of the (strictly convex) boundary. 
The iteration of this step requires the strict convexity assumption imposed on the foliation. In the second step, the uniqueness region is extended in time by means of non-characteristic foliation and Tataru's unique continuation techniques \cite{TATARU}. 

Section \ref{sec:stab} is devoted to the stability analysis of the inverse problem. There, we present a standard geometric optics construction of microlocal approximate solutions (parametrices). As a consequence of this analysis, we obtain continuity estimates for the observation map and the ellipticity of a normal pseudodifferential operator associated with the microlocal back-propagation of the boundary data. 

The paper concludes with an appendix, where we provide the detailed derivation of one the Carleman estimates employed in the proof of the unique continuation result.


\section{Unique Continuation and Injectivity of the Measurement Operator}\label{sec:UC}

In this section we prove the injectivity of the measurement operator $\Lambda$ by establishing a unique continuation property for solutions of the Moore--Gibson--Thompson equation. 
More precisely, we show that if the boundary measurement $\Lambda f_3$ vanishes, then the corresponding solution of the forward problem must vanish identically, which in turn implies that the unknown initial datum $f_3$ is zero. The proof of this unique continuation result constitutes the core analytical ingredient in the derivation of injectivity for the inverse problem.

In this section, we will work with the following generalization of the MGT operator defined in \eqref{eq:MGT-operator}:
\begin{equation}\label{general-mgt-operator}
P u := \, u_{ttt} + \alpha(x)\, u_{tt} - c^2(x)\, \partial_i\!\big(g^{ij}(x)\partial_j u\big) - \, \partial_i\!\big(g^{ij}(x)\partial_j u_t\big),
\end{equation}
with $g^{ij}(x)$ a smooth, symmetric, and positive definite matrix. Up to lower order terms, we obtain \eqref{eq:MGT-operator} from \eqref{general-mgt-operator} by choosing $g = \frac{b}{\tau}dx^2$ and relabeling coefficients.

Although the inverse problem is originally formulated for the special case corresponding to the Euclidean Laplacian, the unique continuation arguments are developed in this more general geometric setting. This choice allows us to exploit available Carleman estimates for wave operators with variable coefficients established in the literature \cite{BENJAMIN-ACOSTA-THERMO}, which we explicitly invoke in the proof, while at the same time producing a result that applies to a wider class of models. Moreover, we believe that presenting in detail the Carleman computations adapted to this operator is of independent interest and value for the general literature; for this reason, the full derivation is included in the appendix. This makes the analysis self-contained and provides a reusable framework for related problems.

Below, for any $\eta\in[0,T]$ we use the notation
$$
\begin{gathered}
\Omega^{\eta}:=\{x \in \Omega: \operatorname{dist}(x, \partial \Omega)<\eta\}
\quad \text { and } \quad Q^{\eta}=\left\{(x, t) \in Q: x \in \Omega^{\eta}, 0<t<T-\operatorname{dist}(x, \partial \Omega)\right\};
\end{gathered}
$$
and near $\partial\Omega$, we consider boundary normal coordinates 
\[
x^n=\operatorname{dist}_g(x,\partial\Omega),
\]
where
\[
\operatorname{dist}_g(x,\partial\Omega)
=
\inf_{y\in\partial\Omega} d_g(x,y),
\]
with $d_g$ denoting the Riemannian distance associated with the metric $g$. The mains results of this section are the following.
\begin{theorem}\label{Theorem-unique-continuation}[Boundary unique continuation for MGT]
Suppose that $\Omega$ is a strictly convex domain for the metric $g=\left(g_{i j}(x)\right)_{i j}$; $\alpha $, $c$ are smooth functions and $B_1,B_2$ are smooth vector fields, all depending only on the variable $x$.  In addition, assume that $T>0$ is such that $x^n=\operatorname{dist}_g(x, \partial \Omega)$ is a smooth function in $\Omega$ with non-zero differential for $0 \leq x^n \leq T$ and the level surfaces $\left\{x^n=s\right\}$, for $0 \leq s \leq T$, are strictly convex for the metric $g$. If
\begin{equation}\label{MGT-System}
\begin{gathered}
u_{t t t}+\alpha u_{t t}-c^2 \partial_{x^j}\left(g^{i j}(x) \partial_{x^i} u\right)-\partial_{x^j}\left(g^{i j}(x) \partial_{x^i} u_t\right)+ B_1 \nabla u_t + B_2 \nabla u  = 0\quad \text { in } Q \\
u(x, 0)=0, \quad u_t(x, 0)=0, \quad u_{tt}(x,0) = f\\
\left.u\right|_{\Sigma}=0,\left.\quad \dfrac{\partial u}{\partial \nu}\right|_{\Sigma}=0,
\end{gathered}
\end{equation}
with $f \in C^{\infty}(\bar{\Omega})$ and $\operatorname{supp} f \subset \Omega^T$, then $u=0$ in $Q^T$. In particular, $f=0$ in $\Omega^T$.
\end{theorem}

The proof of this theorem will be presented in subsections \ref{subsec:prelim_UC}-\ref{subsec:end_proof}. We first state and demonstrate the associated uniqueness result for the inverse problem.

\begin{corollary}\label{cor:uniqueness}[Uniqueness of the inverse problem]
Let $\Omega$ be a strictly convex domain and $T>0$ such that  $x^n=\operatorname{dist}_g(x,\partial\Omega)$ is smooth with strictly convex level sets $\{x^n=s\}$, $0\le s\le T$, for $g=\frac{\tau}{b}dx^2$. Then $\Lambda f=0$ implies $f=0$ in $\Omega_T$.
\end{corollary}

\begin{proof}[Proof of the Corollary]
Suppose that $\Lambda f=0$. In order to use Theorem \ref{Theorem-unique-continuation}, we consider the restriction of $u$ to the exterior domain
\[
Q_{\mathrm{ext}}
=
\bigl(\mathbb{R}^{n}\setminus\overline{\Omega}\bigr)\times(0,T) .
\]
Since $\operatorname{supp}(f)\subset \Omega$, the exterior function \[ u_{\mathrm{ext}} = u\big|_{Q_{\mathrm{ext}}} \] satisfies homogeneous initial conditions. Moreover, the assumption $\Lambda f=0$ yields homogeneous Dirichlet data on $\Sigma$. Therefore, $u_{\mathrm{ext}}$ solves
\[
\begin{cases}
Pu_{\mathrm{ext}}=0,
& \text{in } Q_{\mathrm{ext}},\\[1ex]
u_{\mathrm{ext}}|_{t=0}=0,
\qquad
\partial_tu_{\mathrm{ext}}|_{t=0}=0,
\qquad
\partial_t^2u_{\mathrm{ext}}|_{t=0}=0,
& \text{in } \mathbb{R}^{n}\setminus\overline{\Omega},\\[1ex]
u_{\mathrm{ext}}=0,
& \text{on } \Sigma .
\end{cases}
\]
By uniqueness for the exterior problem, $u_{\mathrm{ext}}\equiv 0$ in $Q_{\mathrm{ext}}$, and consequently $\partial_\nu u_{\mathrm{ext}}=0$ on $\Sigma$. Then, all the assumptions of Theorem \ref{Theorem-unique-continuation} are satisfied and applying the theorem yields the result.
\end{proof}

\subsection{Preliminaries for the proof of Theorem \ref{Theorem-unique-continuation}}\label{subsec:prelim_UC}

Before entering the proof, we briefly describe the strategy used to handle the third--order nature of the Moore--Gibson--Thompson equation. 
The key observation --- which is presented below as a lemma --- is that the operator can be factorized in a suitable way using wave--type operators, which allows us to reduce the unique continuation analysis to a sequence of second--order problems. 
This factorization plays a central role in the construction of the Carleman estimates employed below.

Let us first introduce the coefficient
$$
\gamma(x):=\alpha(x)-c^2(x),
$$
which allows us to write \eqref{general-mgt-operator} as 
\begin{equation}\label{split-operator}
 P u:=P_0 u_t+c^2 P_0 u+\gamma(x) u_{t t},  
\end{equation}
with $P_0$ a wave-type operator given by
$$
P_0:= \partial_t^2-\frac{\partial}{\partial x^j}\left(g^{i j}(x) \frac{\partial}{\partial x^i}\right) .
$$

Let's recall that for any $y \in \partial \Omega$, we can prescribe boundary normal coordinate $\left(x^{\prime}, x^n\right)$ near $y$, satisfying that $\Omega$ and $\partial \Omega$ are given by locally and respectively by $x^n>0$ and $x^n=0$. Moreover, in such coordinates, the metric $g$ takes the form 
\begin{equation}\label{definition-metric}
   g_{\alpha, \beta}\left(x^{\prime}, x^n\right) \mathrm{d} x^\alpha \mathrm{d} x^\beta+\left(\mathrm{d} x^n\right)^2 
\end{equation}
for $\alpha, \beta \leq n-1$. We denote $\tilde{g}=\left(g_{\alpha \beta}(x)\right)$. 

In this coordinates, the strict convexity of the level surfaces $\left\{x^n=s\right\}$ translates into
$$
\Pi(v, v)=\left(-\dfrac{1}{2} \dfrac{\partial g_{\alpha \beta}}{\partial x^n}\right) v^\alpha v^\beta \geq \kappa_s|v|_{\tilde{g}}^2, \quad \forall v \in T\left\{x^n=s\right\}
$$
with $\kappa_s=R_s^{-1}>0$ the smallest eigenvalue (principal curvature) of the second fundamental form $\Pi$ in $\left\{x^n=s\right\}$, and $R_s$ the largest curvature radius of $\left\{x^n=s\right\}$. The analogous condition for convectors follows from the natural isomorphism $\xi_i=g_{i j}(x) v^j$ and reads

\begin{equation}\label{positive-definiteness-inequality}
    \Pi(\xi, \xi)=\left(\dfrac{1}{2} \dfrac{\partial g^{\alpha \beta}}{\partial x^n}\right) \xi_\alpha \xi_\beta \geq \kappa_s|\xi|_{\tilde{g}}^2, \quad \forall \xi \in T^*\left\{x^n=s\right\}.
\end{equation}

We also define the following Carleman-weight function and some sets related to it: for $0<a<1$ and $R,r>0$, we introduce
$$
\varphi(x, t):=\left(R-x^n\right)^2-a t^2-r^2,
$$
and the set
\[
Q_0:=(\Omega \times[0, \infty)) \cap\left\{\varphi>0\right\},\quad \Omega_0:= \Omega\cap\left\{\varphi>0\right\}\quad\text{and} \quad\Gamma_0=Q_0\cap\{\varphi=0\};
\]
in addition, given a point $y=\left(y^{\prime}, 0\right) \in \partial \Omega$ and local coordinates $\left(U,\left(x^1, \ldots, x^{n-1}\right)\right)$ in a boundary neighborhood of $y^{\prime}$, we define the local sets

\begin{equation}\label{definition-local-spaces}
\begin{aligned}
& Q_0^y:=\left\{(t, x) \in Q: \varphi(x, t)>0, \;x^{\prime} \in U\right\}, \\
& \Gamma_0^y:=Q_0^y\cap\{\varphi=0\},\\
& \Omega_0^y:=\left\{x \in \Omega:\varphi(x, 0)>0, \;x^{\prime} \in U\right\}.
\end{aligned}
\end{equation}

\begin{lemma}\label{bound for both norms}
For initial conditions as in \eqref{MGT-System},
there is a constant $C>0$ such that
\begin{equation}\label{lemma-bound-on-both-norms}
\left\|P_0 u_t\right\|_\tau^2+\left\|P_0 u\right\|_\tau^2 \leq \|P u\|_\tau^2+\left\|\gamma u_{tt}\right\|_\tau^2 +C\int_{\Omega_0}|f|^2 e^{2 \tau \varphi(\cdot, 0)} d x .
\end{equation}
Here, $\|\cdot\|_{\tau}$ corresponds to the following weighted norm
$$
\|w\|_{\tau}^2:=\int_{Q_0} e^{2 \tau \varphi}|w|^2 d x d t.
$$
\end{lemma}
\begin{proof}
According to \eqref{split-operator}, we see that
\begin{equation}\label{equality-before-integration-by-parts}
\left\|P u-\gamma u_{t t}\right\|_{\tau}^2=\left\|P_0 u_t\right\|_{\tau}^2+\left\|c^2P_0 u\right\|_{\tau}^2+\int_{Q_0} c^2 e^{2\tau \varphi} \partial_t\left|P_0 u\right|^2 d x d t.
\end{equation}
We want to estimate from below the third term of the RHS of \eqref{equality-before-integration-by-parts}. Consider a smooth and finite partition of unity $\left\{\chi_i\right\}$ subordinate to a collection $\left\{U_i\right\}$, and consider points $y \in {U_i}$. Then, since $\varphi_t \leq 0$ on $Q_0$ and given the specific form of initial conditions, we obtain

\begin{align*}
\int_{Q_0} c^2 e^{2\tau \varphi} \partial_t\left|P_0 u\right|^2 d t d x &= \Sum_i\int_{Q_0^y} \chi_ic^2 e^{2\tau \varphi} \partial_t\left|P_0 u\right|^2 d t d x  \\
&=\Sum_i\left(\int_{\Omega_0^y} \chi_ic^2\left|P_0 u(\cdot, t(x))\right|^2 e^{2 \tau \varphi(\cdot, t(x))} d x-\int_{\Omega_0^y}\chi_ic^2\left|P_0 u(\cdot, 0)\right|^2 e^{2 \tau \varphi(\cdot, 0)} d x \right.\\
& \left.- 2\int_{\Omega_0^y}\int_0^{t(x)} \chi_i\tau c^2\left|P_0 u\right|^2\left(\partial_t \varphi\right) e^{\tau \varphi} d t d x .\right) \\
&\geq -\Sum_i\int_{\Omega_0^y} \chi_i c^2\left|P_0 u(\cdot, 0)\right|^2 e^{2 \tau \varphi(\cdot, 0)} d x\\
&= -\Sum_i\int_{\Omega_0^y} \chi_i c^2\left| f \right|^2 e^{2 \tau \varphi(\cdot, 0)} d x\\
&=-\int_{\Omega_0} c^2\left| f \right|^2 e^{2 \tau \varphi(\cdot, 0)} d x ,
\end{align*}
for all $t\leq t(x)$, with $t(x)$ being the positive time such that $\varphi(x,t(x)) = 0$. It can be taken to be strictly less than $T$ (by choosing a smaller value of the parameter $a$). 

From \eqref{equality-before-integration-by-parts} and since the coefficients are bounded, we conclude that there exists $C>0$ such that
\begin{equation*}
\left\|P_0 u_t\right\|_\tau^2+\left\|c^2P_0 u\right\|_\tau^2 \leq C\left(\|P u\|_\tau^2+\left\|\gamma u_{t t}\right\|_\tau^2 +  \int_{\Omega_0}\left|f\right|^2 e^{2\tau \varphi(\cdot, 0)} d x\right). 
\end{equation*}
\end{proof}

\begin{remark}
The above estimate requires the assumption $f_1 = 0$ on the initial conditions. The treatment of general initial data will certainly necessitate refined estimates adapted to the full MGT operator and is therefore left for future investigation.
\end{remark}


\subsection{Small time unique continuation (step 1)}\label{STEP-1}
In this first step, we establish a local unique continuation property for small times, associated with the Moore--Gibson--Thompson equation. The proof of this result is based on suitable Carleman estimates adapted to the Moore--Gibson--Thompson operator using factorization \eqref{split-operator}. 

\begin{lemma}\label{lemma local unique continuation}
Suppose that the conditions of Theorem \ref{Theorem-unique-continuation} are satisfied. There exists $\varepsilon_1, \varepsilon_2>0$ such that the solution $u$ to \eqref{MGT-System} vanishes in $\Omega^{\varepsilon_1} \times\left[0, \varepsilon_2\right]$.
\end{lemma}
\begin{proof}
Recall the factorization \eqref{split-operator},
\[
Pu=P_0u_t+c^2P_0u+\gamma u_{tt}.
\]
Since $u(\cdot,0)$ and $u_t(\cdot,0)$ vanish identically, their
spatial derivatives also vanish at $t=0$. Evaluating
\eqref{MGT-System} at $t=0$ therefore gives
\[
u_{ttt}(x,0)=-\alpha(x)f(x)=-\alpha(x)u_{tt}(x,0).
\]

We apply the wave Carleman estimate from Lemma \ref{lemma-local-carleman-benjamin} to $u$, $u_t$, and
$u_{tt}$. The first two functions vanish at $t=0$, while $u_{tt}$
satisfies the initial condition $v_t(x,0)=\rho(x)v(x,0)$ with
$\rho=-\alpha$. The lateral boundary conditions for $u_t$ and
$u_{tt}$ follow by differentiating those of $u$ in time. Denoting
the boundary term in \eqref{proof-local-carleman-wave-equation}
by $G_1(u)$, $G_2(u_t)$ and $G_3(u_{tt})$, respectively, we add the three estimates to obtain
\begin{equation}\label{mgt-sum-wave-carleman}
\begin{aligned}
&\int_{Q_0}\bigl(\tau^3|u|^2+\tau|u'|^2\bigr)
  e^{2\tau\varphi}\,dx\,dt\\
&\quad+\int_{Q_0}\bigl(\tau^3|u_t|^2+\tau|u_t'|^2\bigr)
  e^{2\tau\varphi}\,dx\,dt\\
&\quad+\int_{Q_0}\bigl(\tau^3|u_{tt}|^2+\tau|u_{tt}'|^2\bigr)
  e^{2\tau\varphi}\,dx\,dt
  +G_1(u)+G_2(u_t)+G_3(u_{tt})\\
&\le C\bigl(
  \|P_0u\|_\tau^2+\|P_0u_t\|_\tau^2+\|P_0u_{tt}\|_\tau^2
  \bigr).
\end{aligned}
\end{equation}

The third interior integral in \eqref{mgt-sum-wave-carleman}
contains both $\tau^3|u_{tt}|^2$ and $\tau|u_{ttt}|^2$, which we use to control the initial trace $u_{tt}(x,0)$. Indeed, we apply the fundamental theorem of calculus to the spatial integral of
$e^{2\tau\varphi(x,t)}|u_{tt}(x,t)|^2$ on the time interval 
$[0,T]$. Since $\varphi(x,T)<0$ for all $x\in \Omega_0$ and because $\varphi_t$ is bounded in $\overline{\Omega}_0\times[0,T]$,
\begin{equation}\label{mgt-initial-trace}
\begin{aligned}
\int_{\Omega_0}|f|^2e^{2\tau\varphi(x,0)}\,dx
&=-\int_{\Omega_0\times(0,T)}\partial_t\!\left(
    e^{2\tau\varphi}|u_{tt}|^2\right)\,dx\,dt
  +\int_{\Omega_0}|u_{tt}(x,T)|^2\,dx\\
&\le C\int_{Q_0}e^{2\tau\varphi}
   \left(\tau|u_{tt}|^2+\tau^{-1}|u_{ttt}|^2\right)\,dx\,dt
   +C_u(\tau).
\end{aligned}
\end{equation}
Here, the term $C_u(\tau)$ contains the trace at $t=T$ and an integral over $(\Omega_0\times(0,T))\setminus Q_0$. Since $\varphi<0$ in this region, we can find a constant $M>0$ such that $|C_u(\tau)|\leq M\tau^3$.

We estimate the right-hand side of \eqref{mgt-sum-wave-carleman}
by applying \eqref{lemma-bound-on-both-norms} to $u$ and $u_t$,
with the initial term in each application given 
$P_0$ applied to the corresponding function at time $t=0$, namely,
\[
P_0u(x,0)=f(x),\qquad
P_0u_t(x,0)=u_{ttt}(x,0)=-\alpha(x)f(x).
\]
Since $\alpha$ is bounded, we obtain
\begin{equation}\label{eq:bound_by_MGT}
\begin{aligned}
\|P_0u_t\|_\tau^2+\|P_0u\|_\tau^2
&\le C\left(
  \|Pu\|_\tau^2+\|\gamma u_{tt}\|_\tau^2
  +\int_{\Omega_0}|f|^2e^{2\tau\varphi(x,0)}\,dx
  \right),\\
\|P_0u_{tt}\|_\tau^2+\|P_0u_t\|_\tau^2
&\le C\left(
  \|P(u_t)\|_\tau^2+\|\gamma u_{ttt}\|_\tau^2
  +\int_{\Omega_0}|f|^2e^{2\tau\varphi(x,0)}\,dx
  \right).
\end{aligned}
\end{equation}
Their sum controls all three wave-operator norms in
\eqref{mgt-sum-wave-carleman}. We now bring together \eqref{mgt-sum-wave-carleman} and \eqref{eq:bound_by_MGT}. Absorbing terms with the left hand side, by taking $\tau$ sufficiently large, gives
\[
\begin{aligned}
&\int_{Q_0}\bigl(\tau^3|u|^2+\tau|u'|^2\bigr)
  e^{2\tau\varphi}\,dx\,dt\\
&\quad+\int_{Q_0}\bigl(\tau^3|u_t|^2+\tau|u_t'|^2\bigr)
  e^{2\tau\varphi}\,dx\,dt\\
&\quad + \int_{Q_0}\bigl(\tau^3|u_{tt}|^2+\tau|u_{tt}'|^2\bigr)
  e^{2\tau\varphi}\,dx\,dt\\
&\lesssim 
  \|Pu\|_\tau^2+\|Pu_t\|_\tau^2 - \left(G_1(u)+G_2(u_t)+G_3(u_{tt})\right) + \int_{\Omega_0}|f|^2e^{2\tau\varphi(x,0)}\,dx.
\end{aligned}
\]
In particular, and using the fact that
\[
\tau|u_{tt}|^2+\tau^{-1}|u_{ttt}|^2
=\tau^{-2}\bigl(\tau^3|u_{tt}|^2+\tau|u_{ttt}|^2\bigr),
\]
along with \eqref{mgt-initial-trace}, we get
\[
\begin{aligned}
&\int_{\Omega_0}|f|^2e^{2\tau\varphi(x,0)}\,dx-  C_u(\tau)\\
&\lesssim \tau^{-2}\left(\|Pu\|_\tau^2+\|Pu_t\|_\tau^2 - \left(G_1(u)+G_2(u_t)+G_3(u_{tt})\right) + \int_{\Omega_0}|f|^2e^{2\tau\varphi(x,0)}\,dx\right) 
\end{aligned}
\]
Adding up the last two inequalities and absorbing terms with the parameter $\tau$ we deduce the inequality
\begin{equation}\label{final-estimate}
\begin{aligned}
&\int_{\Omega_0}|f|^2e^{2\tau\varphi(x,0)}\\
&+\quad \int_{Q_0}\bigl(\tau^3|u|^2+\tau|u'|^2\bigr)
  e^{2\tau\varphi}\,dx\,dt\\
&\quad+\int_{Q_0}\bigl(\tau^3|u_t|^2+\tau|u_t'|^2\bigr)
  e^{2\tau\varphi}\,dx\,dt\\
&\quad+ \int_{Q_0}\bigl(\tau^3|u_{tt}|^2+\tau|u_{tt}'|^2\bigr)
  e^{2\tau\varphi}\,dx\,dt\\
&\lesssim \|Pu\|_\tau^2 + \|Pu_t\|_\tau^2 -G_1(u)-G_2(u_t)-G_3(u_{tt}) + C_u(\tau).
\end{aligned}
\end{equation}
Notice that the integral involving $f$ on the right-hand side has been absorbed since it contained a factor $\tau^{-2}$. By absorbing terms one more time, we can replace the first two terms on the right with $\|Pu+B_1\nabla u_t+B_2\nabla u\|_\tau^2$ and $\|Pu_t+B_1\nabla u_{tt}+B_2\nabla u_t\|_\tau^2$, respectively.

Since $u$ solves \eqref{MGT-System}, both operator expressions on
the right of \eqref{final-estimate} vanish. The second vanishes since it is the
time derivative of the first (the coefficients depend only
on $x$). By shrinking the integration region on the left-hand integral to $\Omega_0\cap\{\varphi(x,0)>\delta\}$, for some $\delta>0$ sufficiently small, we can bound $e^{2\tau\varphi(x,0)}\geq e^{2\tau\delta}$ there, and obtain the inequality
\[
\int_{\Omega_\delta}|f|^2dx\lesssim e^{-2\tau\delta}\left(C_u(\tau)-G_1(u)-G_2(u_t)-G_3(u_{tt})\right). 
\]
Since all the terms inside the parentheses have only polynomial dependence on $\tau$, letting $\tau\to\infty$ gives $f=0$ in $\Omega_\delta$. Finally, we choose
$\varepsilon_1,\varepsilon_2>0$ sufficiently small that
$\Omega^{\varepsilon_1}\times[0,\varepsilon_2]\subset Q_0$.
This proves the lemma.
\end{proof}
\subsection{Unique continuation up to the characteristic set (step 2)}
Our goal now is to prove that $u$ vanishes in the same spatial neighborhood but for times up to the characteristic surface $T-t=\operatorname{dist}(x, \partial \Omega)$. This is achieved in Lemma \ref{lemma continuation of the zero}, whose proof is based on Lemma \ref{Lemma-local-continuation}.

\begin{lemma}\label{lemma continuation of the zero}
Under the condition of Theorem \ref{Theorem-unique-continuation}, $u=0$ in

\[
Q^{\epsilon_1}=\left\{(x, t) \in Q: x \in \Omega^{\epsilon_1}, 0<t<T-\operatorname{dist}(x, \partial \Omega)\right\},
\]
where $\epsilon_1$ is as in Lemma \ref{lemma local unique continuation}.
\end{lemma}

As mentioned before, this lemma follows from a preliminary result, which established a local unique continuation across a non-characteristic hypersurface. When working locally, and similarly to the previous section, we denote by $x^n$ the signed distance to $\partial \Omega$, which, according to the hypothesis of Theorem \ref{Theorem-unique-continuation}, is well defined at least up to $x^n=T$. Close to $\partial \Omega$ and for a sufficiently small $\epsilon>0$, we define the weight function
\[
\psi(x, t):=\left(\epsilon-x^n\right)\left(T-t-x^n\right),
\]
and for $\gamma>0$, we consider the sets
\[
Q_\gamma^{\epsilon}:=\left\{(x, t) \mid \psi(x, t)>\gamma, x^n<\epsilon\right\} \cap Q.
\]
which are such that, $Q^\epsilon_\gamma\subset Q^\epsilon_{\gamma'}$ if $0\leq\gamma'<\gamma$, and they exhaust 
\[
Q^{\epsilon}=\left\{(x, t) \in Q \mid x^n<\epsilon, 0<t<T-x^n\right\},
\]
this is, $Q^{\epsilon}=$ $\bigcup_{\gamma>0} Q_\gamma^{\epsilon}$. Moreover, there exists $\gamma_0>0$ such that
\begin{equation}\label{gamma_0}
\emptyset \neq Q_{\gamma_0}^{\epsilon} \subset\left[0, \epsilon_1\right] \times\left\{x \in \Omega: \operatorname{dist}(x, \partial \Omega)<\epsilon\right\}
\end{equation}
We also denote by $B\left(x_0, t_0 ; r\right)$ the ball centered at $\left(x_0, t_0\right)$ and radius $r$ for the euclidean metric.
\begin{lemma}\label{Lemma-local-continuation}
Suppose that for $\left(x_0, t_0\right) \in Q^{\epsilon}$ and $u$ vanishes below the level surface $\left\{\psi(x, t)=\psi\left(x_0, t_0\right)\right\}$ near $\left(x_0, t_0\right)$, this is, in $Q_{\psi\left(x_0, t_0\right)}^{\epsilon} \cap B\left(x_0, t_0 ; r\right)$ for some $r>0$. Then, $u=0$ in a neighborhood of $\left(x_0, t_0\right)$. 
\end{lemma}

The main technical ingredient behind the proof of this lemma is a pseudodifferential-weighted Carleman estimate for analytically principally normal operators, developed by Tataru \cite{TATARU}. Before demonstrating the lemma, we recall some definitions. 

Let $A \subset \mathbb{R}^d=\mathbb{R}^{d_a} \times \mathbb{R}^{d_b}$, where $d_a+d_b=d$. We denote by $x=\left(x_a, x_b\right)$ the global variables and $\xi=\left(\xi_a, \xi_b\right)$ the associated dual variables. The differential operators involved are assumed to be analytic with respect to the variables $x_a$. We start with the definition of an analytically principally normal operator, followed by the definition of a strongly pseudoconvex function.

\begin{definition}\textbf{(Analytically Principally Normal Operator).} 
Let $P$ be a partial differential operator on an open set $A \subset \mathbb{R}^{d_a} \times \mathbb{R}^{d_b}$ of order $m \in \mathbb{N}^*$, with smooth coefficients and principal symbol $p\left(x_a, x_b, \xi_a, \xi_b\right)$. We say that $P$ is an analytically principally normal operator in $\left\{\xi_a=0\right\}$ inside $A$ if the coefficients of $P$ are real-analytic in the variable $x_a$ and for any $x^0 \in A$ there exist $\Omega_a \subset \mathbb{R}^{d_a}, \Omega_b \subset \mathbb{R}^{d_b}$ such that $x^0 \in \Omega_a \times \Omega_b\subset A$, and there exists a complex neighborhood $\Omega_a^{\mathbb{C}}$ of $\Omega_a$ in $\mathbb{C}^{d_a}$ and a constant $C>0$ such that for all $z_a, \tilde{z}_a \in \Omega_a^{\mathbb{C}}$ and all $\left(x_b, \xi_b\right) \in \Omega_b \times \mathbb{R}^{n_b}, \xi_b \neq 0$, we have
\[ 
\begin{gathered}
\left|\left\{p\left(z_a, \cdot, 0, \cdot\right), p\left(\tilde{z}_a, \cdot, 0, \cdot\right)\right\}\left(x_b, \xi_b\right)\right|+\left|\left\{\overline{p\left(z_a, \cdot, 0, \cdot\right)}, p\left(\tilde{z}_a, \cdot, 0, \cdot\right)\right\}\left(x_b, \xi_b\right)\right| \leq C\left|p\left(z_a, x_b, 0, \xi_b\right)\right|\left|\xi_b\right|^{m-1} \\
\left|\partial_{z_a} p\left(z_a, x_b, 0, \xi_b\right)\right| \leq C\left|p\left(z_a, x_b, 0, \xi_b\right)\right|
\end{gathered}
\] 
\end{definition}
\begin{remark}\label{remark-for-analytic-principally-normal-operator}
Special cases of operators being analytically principally normal in $\left\{\xi_a=0\right\}$ (as considered, for instance, in \cite{HORMANDER-PRINCIPALLY-NORMAL}) are those with analytic coefficients in $x_a$, principally normal and invariant with respect to the null bicharacteristic flow in $\left\{\xi_a=0\right\}$, i.e. satisfying
\[
\left|\{\bar{p}, p\}\left(x_a, x_b, 0, \xi_b\right)\right| \leq C\left|p\left(x_a, x_b, 0, \xi_b\right)\right|\left|\xi_b\right|^{m-1} \quad \text { and } \quad \partial_{x_a} p\left(x_a, x_b, 0, \xi_b\right)=0.
\]
\end{remark}

In the following definitions, we need the conjugated symbol $p_\psi(x, \xi):=p(x, \xi+i \tau \nabla \psi)$ and the Poisson Bracket
\[
\{f, h\}:=\sum_{j=1}^n \frac{\partial f}{\partial \xi_j} \frac{\partial h}{\partial x_j}-\frac{\partial f}{\partial x_j} \frac{\partial h}{\partial \xi_j}.
\]

\begin{definition}\label{definition-pseudoconvex-surface}\textbf{(Strongly pseudoconvex surface)}
Let $A \subset \mathbb{R}^n, \Gamma$ be a closed conic subset of $T^* A$, and let $P$ be principally normal in $\Gamma$ inside A$\Omega$ with principal symbol $p$. Let $S$ be a $C^2$ oriented hypersurface of $A$ and $x^0 \in S \cap A$. We say that $S$ is strongly pseudoconvex in $\Gamma$ at $x^0$ for $P$ if there exists $\psi \in C^2(\Omega ; \mathbb{R})$ such that $S=\{\psi=0\}, \nabla \psi\left(x^0\right) \neq 0$, satisfying:

\[
\begin{aligned}
\operatorname{Re}\{\bar{p},\{p, \psi\}\}\left(x^0, \xi\right)>0, \quad \text { if } p\left(x^0, \xi\right) & =\{p, \psi\}\left(x^0, \xi\right)=0 \text { and } \xi \in \Gamma_{x^0}, \xi \neq 0; \\
\frac{1}{i \tau}\left\{\bar{p}_\psi, p_\psi\right\}\left(x^0, \xi\right)>0, \quad \text { if } p_\psi\left(x^0, \xi\right) & =\left\{p_\psi, \psi\right\}\left(x^0, \xi\right)=0 \text { and } \xi \in \Gamma_{x^0}, \tau>0.
\end{aligned}
\]

\end{definition}

\begin{definition}\label{definition-pseudoconvex-function}\textbf{(Strongly pseudoconvex function)}
Let $P$ be a principally normal operator in $A \subset \mathbb{R}^n$, with principal symbol $p$, let $\varphi \in C^2(\Omega ; \mathbb{R})$ and $\Gamma$ be a closed conic subset of $T^* A$. Let $x^0 \in A$. We say that $\varphi$ is strongly pseudoconvex in $\Gamma$ at $x^0$ for $P$ if:

\[
\begin{aligned}
\operatorname{Re}\{\bar{p},\{p, \varphi\}\}\left(x^0, \xi\right)>0, & \text { if } p\left(x^0, \xi\right)=0 \text { and } \xi \in \Gamma_{x^0}, \xi \neq 0; \\
\frac{1}{i \tau}\left\{\bar{p}_\varphi, p_\varphi\right\}\left(x^0, \xi\right)>0, & \text { if } p_\varphi\left(x^0, \xi\right)=0 \text { and } \xi \in \Gamma_{x^0}, \tau>0.
\end{aligned}
\]
\end{definition}
\begin{remark}\label{remark-surface-implies-pseudoconvex-function}
Given a function $\psi$ satisfying the geometric conditions in Definition \ref{definition-pseudoconvex-surface}, there exists a positive constant $\lambda_0$ large enough such that, for all $\lambda\geq \lambda_0$, $\varphi=e^{\lambda \psi}$ satisfies the stronger pseudoconvexity conditions of Definition \ref{definition-pseudoconvex-function} (see \cite[Proposition 28.3.3]{hormander1994}). 
\end{remark}

\begin{proof}[Proof of Lemma \ref{Lemma-local-continuation}]
We start the proof by showing that the surface $S =\{ \psi = 0\}$ is strongly pseudoconvex for

\[
\mathcal{P}=\partial_t^3 + \alpha \partial_t^2 - \partial_t \sum_{i j} \partial_{x^i}\left(g^{i j} \partial_{x^j}\right)- c^2 \sum_{i j} \partial_{x^i}\left(g^{i j} \partial_{x^j}\right)
\]
on the subset of the cotangent bundle given by $\Gamma$. Here, $g$ is the metric defined in \eqref{definition-metric} and the principal symbol of $\mathcal{P}$  corresponds to $p(x, t ; \xi, \theta)= i\theta \left(-\theta^2+|\xi|_g^2\right)$, where $|\xi|_g^2=g^{i j}(x) \xi_i \xi_j$ is the covector-norm induced by $g$. We point out that the symbol $p$ is analytically principally normal in $\Gamma$ since it trivially satisfies the special case given in Remark \ref{remark-for-analytic-principally-normal-operator}. Now, we need to check both of the conditions on Definition \ref{definition-pseudoconvex-surface}. 

Let $\rho>\beta>0$ be the constants given by the evaluations:
\[
\psi_{x^{\prime}}\left(x_0,t_0\right)=0, \quad\psi_t\left(x_0,t_0\right)=-\beta \quad \text{and} \quad \psi_{x^n}\left(x_0,t_0\right)=-\rho
\]
Also, denoting by $\delta_{i j}$ the Kronecker delta,
\[
\psi_{t t}=0, \quad \psi_{t x^i}=\delta_{i n}, \quad \psi_{x^i x^j}=2 \delta_{i n} \delta_{j n} .
\]
For the first condition, we see that
\[
\{p,\psi \}  = i \theta \sum_k\sum_j \left( g^{jk}+g^{kj}\right) \psi_{x^k} - 3i\theta^2 \psi_t + i \psi_t\abs{\xi}_g^2,
\]
so $p(x_0,t_0,\xi,0) = 0$ and $\{p,\psi \}(x_0,t_0,\xi,0) = 0$ if and only if $\xi = 0$. This implies the set where the inequality should hold is empty, and consequently, the first condition is satisfied by default.

For the second condition, we compute
\[
\begin{aligned}
\left\{p_\psi, \psi\right\}&= \sum_k\left(-2\theta\tau \sum_j g^{kj}\psi_{x^j} - \tau\psi_t\sum_j\left(g^{jk}+g^{kj}\right)\xi_j \right. \\
&\left. +i\theta \sum_j\left(g^{jk}+g^{kj}\right)\xi_j - i2\tau^2\psi_t\sum_jg^{kj}\psi_{x^j}\right)\psi_x \\
&+ \left(6\tau\theta \psi_t -2\tau \sum_{ij}g^{ij}\xi_i\psi_{x^j} - i3\theta^2 + 3i\tau^2\psi_t^2 + i \abs{\xi}^2_g - i\tau^2\abs{\psi_x}^2_g \right)\psi_t.
\end{aligned}
\]
Plugging in $(x^0,\theta,\xi)$, we get
\[
\begin{aligned}
\left\{p_\psi, \psi\right\}(x_0,t_0,\xi,0)&= -\tau \beta\rho \sum_k\sum_j\left(g^{j k}+g^{k j}\right)\xi_j  -2\tau \rho\beta\xi_n \\
&+ i\left( 3 \tau^2 \beta \rho^2-3 \tau^2 \beta^3-\beta|\xi|_g^2 \right).
\end{aligned}
\]
In order for $\mathfrak{Im}\left\{p_\psi, \psi\right\}(x_0,t_0,\xi,0)$ to vanish for all large $\tau$, then necessarily we must have $\rho^2-\beta^2=0$ and $\xi=0$. We conclude that the second condition is satisfied by default as well.

Finally, as mentioned in Remark \ref{remark-surface-implies-pseudoconvex-function}, for a large enough $\lambda>0$, the function $\varphi = e^{\lambda\psi}$ is a strongly pseudoconvex weight function with respect to $\mathcal{P}$ in the set $\Gamma =\left\{(\xi, 0) \in T_{\left(x_0, t_0\right)}^* \Omega\times [0,T]\right\}$. The proof is then complete after applying \cite[Theorem 1]{TATARU}.
\end{proof}

We are now in position to prove Lemma \ref{lemma continuation of the zero}.
\begin{proof}[Proof Lemma \ref{lemma continuation of the zero}] Let $\gamma^*$ be the infimum of the $\gamma>0$, such that
\[
\operatorname{supp} u \cap Q_\gamma^{\epsilon_1}=\emptyset, \quad\forall \gamma>\gamma^*.
\]
Notice that due to \eqref{gamma_0}, $\gamma<\gamma_0$. On the other hand,
if $\gamma^*=0$, we are done.  However, if $\gamma^*>0$, we must have that $\operatorname{supp} u \cap \overline{Q^{\epsilon_1}_{\gamma^*}} \neq \emptyset$ and we
might applying Lemma \ref{Lemma-local-continuation} on every contact point $\left(t^*, x^*\right) \in \operatorname{supp} u \cap Q^{\epsilon_1}_{\gamma^*}$. Then, by compactness, we can find $\gamma_1\in [0,\gamma^*)$ such that $\operatorname{supp} u \cap Q_\gamma^{\epsilon_1}=\emptyset$ for all $\gamma>\gamma_1$, which contradicts the definition of $\gamma^*$. In consequence, we deduce that $u=0$ on every $Q_\gamma^{\epsilon_1}$ with $\gamma>0$, and therefore $u=0$ in $Q^{\epsilon_1}$.
\end{proof}
\subsection{Conclusion of the proof of Theorem \ref{Theorem-unique-continuation}}\label{subsec:end_proof}

By Steps 1 and 2, we have shown that if $u$ has null Cauchy data on $\partial \Omega \times (0,T)$, then $u=0$ in $Q^\varepsilon$ for some  $\varepsilon>0$. The same argument holds if one starts by assuming $u=0$ in $Q^\sigma$ with $\sigma\in [0,T]$, that is, $\exists \varepsilon(\sigma)>0$ such that $u=0$ in $Q^{\sigma+\varepsilon}.$


Now consider the interval $[0,T]$. Since it is compact, there is $N\in\mathbb{N}$ and $ \{\sigma_j\}_{j=1}^{N}\subset [0,T]$ with $\sigma_1=0$, such that $ [0,T]\subset [0,\epsilon_1)\cup\bigcup_{j=2}^{N}\left(\sigma_j,\,\sigma_j+\varepsilon(\sigma_j)\right)$.

Starting from null Cauchy data on $\partial \Omega \times (0,T)$, the local propagation result from Steps 1 and 2 extends the vanishing across each one of these finitely many intervals, that is, 
$u=0$ on every $Q^{\sigma_j+\epsilon(\sigma_j)}$,
and in particular,
\[
u=0 \qquad \text{in } Q^T =\{(x,t)\in Q:\text{dist}_g(x,\partial\Omega)+t<T\}.
\]
This completes the proof.

\section{Stability of the Inverse Problem}\label{sec:stab}

In this section, we establish quantitative stability estimates for the inverse problem of recovering the unknown initial datum $f_3$ from the boundary measurement $\Lambda f_3$. 
While the injectivity of the measurement operator is obtained through unique continuation arguments, the present analysis provides explicit continuity bounds that measure the sensitivity of the reconstruction with respect to perturbations in the data.

The proof is based on a microlocal analysis of the forward operator and the construction of geometric optics approximate solutions (parametrices). 
This approach allows us to derive precise mapping properties for the observation operator and to identify a pseudodifferential normal operator associated with the microlocal back-propagation of the boundary data, whose ellipticity, in conjunction with uniqueness, yields the desired stability estimate.

\subsection{Forward parametrix}\label{subsec:fwd_para}

For this section, we go back to our original formulation of the MGT equation with variable coefficients:
\[
\tau(x)\, u_{ttt} + \alpha(x)\, u_{tt} - c^2(x)\, \Delta u - b(x)\, \Delta u_t = 0.
\]
The associated differential operator is
\[
P := \tau\, \partial_t^3 + \alpha\, \partial_t^2 - c^2\, \Delta - b\, \Delta \partial_t,
\]
with full symbol 
\[
p(x,\xi,\sigma) = -i \tau\, \sigma^3 - \alpha\, \sigma^2 + i b\, \sigma\, |\xi|^2 + c^2\, |\xi|^2,
\]
which is independent of time; while its principal symbol is
\[
p_3(x,\xi,\sigma) = -i \tau\, \sigma^3  + i b\, \sigma\, |\xi|^2.
\]

We want to construct an operator $E$ acting on $\mathbf{f}=\left[f_1, f_2,f_3\right]$ such that $P(E \mathbf{f})=0$ microlocally, i.e. $P(E \mathbf{f}) \in$ $C^{\infty}\left(\mathbb{R}^n\right)$, and satisfying also the initial conditions (microlocally):

\begin{equation}\label{iniitial_conditions_parametrix}
\left.E f\right|_{t=0}=f_1,\left.\quad \partial_t E f\right|_{t=0}=f_2 \left.\quad \partial_{t}^2 E f\right|_{t=0}=f_3 , \quad f_1, f_2,f_3 \in \mathcal{D}^{\prime}
\end{equation}
We perform this construction for $f_i \in C^{\infty}\left(\mathbb{R}^n\right)$ with $\operatorname{supp} f_i \subset \bar{\Omega}$, and we later extend the domain of $E$ to compactly supported distributions.\\

Let's start by considering $E$ of the form

\[
E\left(f_1, f_2,f_3\right)(t)=\frac{1}{(2 \pi)^n} \sum_{k= 1}^3 \int e^{i \phi_k(x,t ; \xi)}\left(\sum_{j=1}^3|\xi|^{1-j}a_{j,k}(x,t;\xi)\hat{f}_{j}
\right) d \xi,
\]
where $\hat{f}_j$ is the Fourier transform of $f_i$, for $i=1,2,3$.

We look for phase functions $\phi_{ k} \in C^{\infty}\left(\mathbb{R}_t \times \mathbb{R}_x^n \times\left(\mathbb{R}^n \backslash 0\right)_{\xi}\right)$, positive homogeneous of degree 1 in $\xi$, and such that $\nabla_{x, t, \xi} \phi_{ k} \neq 0$. The amplitude functions $a_{i, k}$ are classical Kohn–Nirenberg symbols of order zero, admitting asymptotic expansions

\[
a_{i, k} \sim \sum_{j=0}^{\infty} a_{i, k}^{(j)}(x,t; \xi), \quad i=1,2,3,
\]
with $a_{i, k}^{(j)}$ smooth functions, positive homogeneous of degree $-j$ in the variable $\xi$. The notation above means that $a_{i, k}-\sum_{j=0}^m a_{i, k}^{(j)} \in$ $S^{-m-1}\left(\mathbb{R} \times \mathbb{R}^n \times \mathbb{R}^n\right)$ for all $m \geq 0$; see \cite{GrigisSjostrand1994}.

The construction is carried out iteratively by applying $P$ to $E \mathbf{f}$ and choosing $\phi_{ k}$and $a_{i, k}^{(j)}$ such that they eliminate the homogeneous terms of every degree. By applying $P$ and arranging terms with respect to the degree of homogeneity, we get 
\begin{equation}
\begin{aligned}
P E \mathbf{f} \sim & \frac{1}{(2 \pi)^n} \sum_{k=1}^3 \int e^{i \phi_k}\left\{\left(-i\tau\left(\partial_t\phi_k\right)^3-ib\partial_t\phi_k|\nabla_x\phi|^2\right)\left(
\sum_{j=1}^3|\xi|^{1-j}a_{j,k}\hat{f}_j
\right)\right. \\
& \left.+\sum_{i=1}^3 \sum_{j=0}^{\infty}\left[ P_2 a_{i, k}^{(j)}+\left(P_1 + P\phi\right) a_{i, k}^{(j-1)}+Pa_{i, k}^{(j-2)}\right]|\xi|^{-(i-1)} \hat{f}_i\right\} d \xi,
\end{aligned}
\end{equation}
where $a_{i,k}^{(-1)} = a_{i,k}^{(-2)} = 0$ for $i,k=1,2,3$, and 
\begin{equation}
\begin{aligned}
P_2=& -2\tau|\partial_t\phi_k|^2\partial_t -\tau\partial_t\phi_k\partial_t^2\phi_k -\tau \partial_t|\partial_t\phi_k|^2-\alpha\left|\partial_t \phi_k\right|^2+c^2\left|\nabla_x \phi_k\right|^2 \\
&+ 2b\partial_t\phi_k\nabla_x\phi_k\nabla_x +b\partial_t\phi_k \Delta\phi_k +b\partial_t|\nabla_x\phi_k|^2 +b|\nabla_x\phi_k|^2\partial_t \\
P_1=& 2i\alpha\partial_t\phi_k\partial_t -2i\tau\partial_t^2\phi_k\partial_t +i\tau \partial_t^2\phi_k\partial_t\\
&-2ic^2\nabla_x\phi_k\nabla_x -2ib\partial_t\nabla_x\phi_k\nabla_x - ib\Delta \phi_k\partial_t\\
&-ib\partial_t\phi_k\Delta -2ib\nabla_x\phi(\partial_t\nabla_x) +3i \tau\partial_t\phi_k\partial_t^2.
\end{aligned}
\end{equation}
Also, considering \eqref{iniitial_conditions_parametrix},
we have the following initial conditions. 
From $\left.E \mathbf{f}\right|_{t=0}=f_1$ we get:
\begin{align*}
\sum_{k=1}^3 a_{1,k}(0,x,\xi) &= 1, \\
\sum_{k=1}^3 a_{2,k}(0,x,\xi) &= 0, \\
\sum_{k=1}^3 a_{3,k}(0,x,\xi) &= 0;
\end{align*}
from $\left.\partial_tE \mathbf{f}\right|_{t=0}=f_2$:
\begin{align*}
\sum_{k=1}^3 \left( i \partial_t\phi_k(0,x,\xi)\, a_{1,k}(0,x,\xi) + \partial_t a_{1,k}(0,x,\xi) \right) &= 0, \\
\sum_{k=1}^3 \left( i \partial_t\phi_k(0,x,\xi)\, a_{2,k}(0,x,\xi) + \partial_t a_{2,k}(0,x,\xi) \right) &= |\xi|, \\
\sum_{k=1}^3 \left( i \partial_t\phi_k(0,x,\xi)\, a_{3,k}(0,x,\xi) + \partial_t a_{3,k}(0,x,\xi) \right) &= 0;
\end{align*}
and from $\left.\partial_t^2E \mathbf{f}\right|_{t=0}=f_3$:
\begin{align*}
\sum_{k=1}^3 \left( -|\partial_t\phi_k(0,x,\xi)|^2\, a_{1,k}(0,x,\xi) + 2 i \partial_t\phi_k(0,x,\xi)\, \partial_t a_{1,k}(0,x,\xi) + \partial_t^2 a_{1,k}(0,x,\xi)  + i a_{1,k}(0,x,\xi)\, \partial_{t}^2\phi_k(0)\right) &= 0, \\
\sum_{k=1}^3 \left( -|\partial_t\phi_k(0,x,\xi)|^2\, a_{2,k}(0,x,\xi) + 2 i \partial_t\phi_k(0,x,\xi)\, \partial_t a_{2,k}(0,x,\xi) + \partial_t^2 a_{2,k}(0,x,\xi)  + i a_{2,k}(0,x,\xi)\, \partial_{t}^2 \phi_k(0)\right)&= 0, \\
\sum_{k=1}^3 \left( -|\partial_t\phi_k(0,x,\xi)|^2\, a_{3,k}(0,x,\xi) + 2 i \partial_t\phi_k(0,x,\xi)\, \partial_t a_{3,k}(0,x,\xi) + \partial_t^2 a_{3,k}(0,x,\xi)  + i a_{3,k}(0,x,\xi)\, \partial_{t}^2 \phi_k(0)\right)&= |\xi|^2.
\end{align*}

To get rid of the terms homogeneous of degree 3, we choose $\phi_k(t,x,\xi)$ to be a solution of the eikonal equation plus initial conditions:

\[
p_3(\partial_t \phi_k, \nabla_x \phi_k) = 0, \quad \phi_k(0,x ; \xi)=x \cdot \xi.
\]
Explicitly, this gives:
\begin{equation}
-i \tau \partial_t \phi_k\left( (\partial_t \phi_k)^2  -\dfrac{b}{\tau}|\nabla_x \phi_k|^2 \right) =0.
\end{equation}
So, we see that there are three possibilities:
\begin{equation}\label{eikonal-wave}
\partial_t \phi_k = 
\begin{cases}
\pm \sqrt{\dfrac{b}{\tau}}\left|\nabla_x \phi_k\right|, & \text{if } k=1,2, \\
0, & \text{if } k=3,
\end{cases}
\qquad \phi_k(0,x;\xi) = x \cdot \xi.
\end{equation}
In particular, we see that $\phi_3(t,x;\xi) = x\cdot\xi$; the other two cases ($k=1,2$) give rise to typical solutions to the eikonal equation associated to the wave speed $\sqrt{b/\tau}$.

We next build the coefficients $a_{1,k}^{(j)}$ recursively by solving a sequence of transport equations stemming from the elimination of homogeneous terms of lower degrees. In the following we write $a_{1,k}(0,x,\xi) = a_{1,k}(0)$ to simplify the notation, and given that $\phi_k$ satisfies \eqref{eikonal-wave}, we set
\begin{equation*}
\rho_1(\xi)= \sqrt{\dfrac{b}{\tau}}|\xi|, \quad \rho_2(\xi)= -\sqrt{\dfrac{b}{\tau}}|\xi| \quad \text{and} \quad \rho_3(\xi) = 0.
\end{equation*}
To deal with the terms homogeneous of degree 2, we choose $a^{(0)}_{1,k}$ such that they satisfy
\begin{equation*}
P_2a_{1,k}^{(0)} = 0,
\end{equation*}
with initial conditions $a_{1,1}^{(0)}(0) =  a_{1,2}^{(0)}(0) = 0$ and $a_{1,3}^{(0)}(0)= 1$, which are the solutions of the algebraic system:

\[
\sum_{k=1}^{3} a_{1,k}^{(0)}(0) = 1,
\quad
\sum_{k=1}^{3} i \rho_k\, a_{1,k}^{(0)}(0) = 0,
\quad
\sum_{k=1}^{3} -\rho_k^2\, a_{1,k}^{(0)}(0) = 0.
\]
Now, for the terms homogeneous of degree 1, we solve
\begin{equation*}
P_2a_{1,k}^{(1)} = -\left(P_1+P\phi \right)a_{1,k}^{(0)}
\end{equation*}
with initial conditions obtained by solving the system:

\begin{align*}
&\sum_{k=1}^{3} a_{1,k}^{(1)}(0) = 0,\\
&\sum_{k=1}^{3} i \rho_k\, a_{1,k}^{(1)}(0) + \partial_t a_{1,k}^{(0)}(0) = 0,\\
&\sum_{k=1}^{3} -\rho_k^2\, a_{1,k}^{(1)}(0) + 2i \rho_k\, \partial_t a_{1,k}^{(0)}(0) +i\, a_{1,k}^{(0)}(0)\, \partial_{t}^2 \phi_k(0) = 0.
\end{align*}
In general, for the terms homogeneous of degree $2-j$ with $j \geq 2$, we proceed by induction, so we choose $a_{1, k}^{(j)}$ satisfying

\begin{equation*}
P_2a_{1,k}^{(j)} = -\left(P_1+P\phi \right)a_{1,k}^{(j-1)} - Pa_{1,k}^{(j-2)},
\end{equation*}
subject to initial conditions obtained by solving:

\begin{align*}
&\sum_{k=1}^{3} a_{1,k}^{(j)}(0) = 0,\\
&\sum_{k=1}^{3} i \rho_k\, a_{1,k}^{(j)}(0) + \partial_t a_{1,k}^{(j-1)}(0) = 0,\\
&\sum_{k=1}^{3} -\rho_k^2\, a_{1,k}^{(j)}(0) + 2i \rho_k\, \partial_t a_{1,k}^{(j-1)}(0) + \partial_t^2 a_{1,k}^{(j-2)}(0) + i\, a_{1,k}^{(j-1)}(0)\, \partial_{t}^2 \phi_k(0) = 0.
\end{align*}

Similarly, for the coefficients $a_{2,k}$ we get the following. The first transport equation is 
\begin{equation*}
P_2a_{2,k}^{(0)} = 0,
\end{equation*}
with initial conditions $a_{2,1}^{(0)}(0) = \dfrac{i\sqrt{\tau}}{2b}, \, a_{2,2}^{(0)}(0) = -\dfrac{i\sqrt{\tau}}{2b}$ and $a_{2,3}^{(0)}(0)= 0$, which are the  solutions of the system:

\[
\sum_{k=1}^{3} a_{2,k}^{(0)}(0) = 0,
\quad
\sum_{k=1}^{3} i \rho_k\, a_{2,k}^{(0)}(0) = |\xi|,
\quad
\sum_{k=1}^{3} -\rho_k^2\, a_{2,k}^{(0)}(0) = 0.
\]
To remove terms homogeneous of degree 1, we solve
\begin{equation*}
P_2a_{2,k}^{(1)} = -\left(P_1+P\phi \right)a_{2,k}^{(0)}
\end{equation*}
with initial conditions obtained from the system:

\begin{align*}
&\sum_{k=1}^{3} a_{2,k}^{(1)}(0) = 0,\\
&\sum_{k=1}^{3} i \rho_k\, a_{2,k}^{(1)}(0) + \partial_t a_{2,k}^{(0)}(0) = 0,\\
&\sum_{k=1}^{3} -\rho_k^2\, a_{2,k}^{(1)}(0) + 2i \rho_k\, \partial_t a_{2,k}^{(0)}(0) + i\, a_{2,k}^{(0)}(0)\, \partial_{t}^2 \phi_k(0) = 0.
\end{align*}
Now, for $j \geq 2$, we solve the transport equation 

\begin{equation*}
P_2a_{2,k}^{(j)} = -\left(P_1+P\phi \right)a_{2,k}^{(j-1)} - Pa_{2,k}^{(j-2)},
\end{equation*}
subject to the initial conditions coming from:

\begin{align*}
&\sum_{k=1}^{3} a_{2,k}^{(j)}(0) = 0,\\
&\sum_{k=1}^{3} i \rho_k\, a_{2,k}^{(j)}(0) + \partial_t a_{2,k}^{(j-1)}(0) = 0,\\
&\sum_{k=1}^{3} -\rho_k^2\, a_{2,k}^{(j)}(0) + 2i \rho_k\, \partial_t a_{2,k}^{(j-1)}(0) + \partial_t^2 a_{2,k}^{(j-2)}(0) + i\, a_{2,k}^{(j-1)}(0)\, \partial_{t}^2 \phi_k(0) = 0.
\end{align*}
Finally, for the coefficients $a_{3,k}$, the first transport equation is
\begin{equation}\label{eq: coefficientsa31}
P_2a_{3,k}^{(0)} = 0,
\end{equation}
with initial conditions $a_{3,1}^{(0)}(0) = a_{3,2}^{(0)}(0) = -\dfrac{\tau}{2b}$ and $a_{3,3}^{(0)}(0)= \dfrac{\tau}{b}$, obtained from:

\begin{equation}\label{initial-conditions-for-principal-term}
  \sum_{k=1}^{3} a_{3,k}^{(0)}(0) = 0,
\quad
\sum_{k=1}^{3} i \rho_k\, a_{3,k}^{(0)}(0) = 0,
\quad
\sum_{k=1}^{3} -\rho_k^2\, a_{3,k}^{(0)}(0) = |\xi|^2.
\end{equation}
The elimination of terms homogeneous of degree 1 implies solving
\begin{equation*}
P_2a_{3,k}^{(1)} = -\left(P_1+P\phi \right)a_{3,k}^{(0)},
\end{equation*}
with initial conditions obtained from
\begin{align*}
&\sum_{k=1}^{3} a_{3,k}^{(1)}(0) = 0,\\
&\sum_{k=1}^{3} i \rho_k\, a_{3,k}^{(1)}(0) + \partial_t a_{3,k}^{(0)}(0) = 0,\\
&\sum_{k=1}^{3} -\rho_k^2\, a_{3,k}^{(1)}(0) + 2i \rho_k\, \partial_t a_{3,k}^{(0)}(0) + i\, a_{3,k}^{(0)}(0)\, \partial_{t}^2 \phi_k(0) = 0;
\end{align*}
while for $j \geq 2$, we solve the transport equations
\begin{equation*}
P_2a_{3,k}^{(j)} = -\left(P_1+P\phi \right)a_{3,k}^{(j-1)} - Pa_{3,k}^{(j-2)},
\end{equation*}
subject to the initial conditions satisfying the system:
\begin{align*}
&\sum_{k=1}^{3} a_{3,k}^{(j)}(0) = 0,\\
&\sum_{k=1}^{3} i \rho_k\, a_{3,k}^{(j)}(0) + \partial_t a_{3,k}^{(j-1)}(0) = 0,\\
&\sum_{k=1}^{3} -\rho_k^2\, a_{3,k}^{(j)}(0) + 2i \rho_k\, \partial_t a_{3,k}^{(j-1)}(0) + \partial_t^2 a_{3,k}^{(j-2)}(0) + i\, a_{3,k}^{(j-1)}(0)\, \partial_{t}^2 \phi_k(0) = 0.
\end{align*}

The previously constructed operator $E(t): C_0^{\infty} \times C_0^{\infty} \times C_0^{\infty} \rightarrow C^{\infty}\left(\left(-\epsilon_0, \epsilon_0\right) \times \mathbb{R}^n\right)$ is called a parametrix for the MGT operator, or more precisely, for the solution operator of the initial value problem \eqref{eq:MGT-operator}-\eqref{eq:forward}. It can be extended to a linear operator from $\mathcal{E}^{\prime}\left(\mathbb{R}^n\right) \times \mathcal{E}^{\prime}\left(\mathbb{R}^n\right)\times \mathcal{E}^{\prime}\left(\mathbb{R}^n\right)$ to $C^{\infty}\left(\left(-\epsilon_0, \epsilon_0\right) ; \mathcal{D}^{\prime}\left(\mathbb{R}^n\right)\right)$. Then, for all $\mathbf{f}=\left[f_1, f_2,f_3\right] \in \mathcal{E}^{\prime}\left(\mathbb{R}^n\right) \times \mathcal{E}^{\prime}\left(\mathbb{R}^n\right)\times \mathcal{E}^{\prime}\left(\mathbb{R}^n\right), u=E(t) \mathbf{f}$ is such that
\[
P u=K \mathbf{f},\left.\quad u\right|_{t=0}=f_1,\left.\partial_t u\right|_{t=0}=f_2,\left.\partial^2_{t} u\right|_{t=0}=f_3,
\]
where 
$K: \mathcal{E}^{\prime}\left(\mathbb{R}^n\right) \times$ $\mathcal{E}^{\prime}\left(\mathbb{R}^n\right) \times \mathcal{E}^{\prime}\left(\mathbb{R}^n\right)\rightarrow C_0^{\infty}\left(\left(-\epsilon_0, \epsilon_0\right) \times \mathbb{R}^n\right)$ is smoothing. By restricting $E\mathbf{f}$ to the boundary, we obtain a microlocal representation of the measurement operator $\Lambda \mathbf{f}$ as an oscillatory integral:
\begin{equation}\label{eq:meas_op_FIO}
\Lambda \mathbf{f} \simeq \frac{1}{(2\pi)^n} \sum_{k=1}^2 \int 
e^{i\phi_{k}(x,t;\xi)}
\left( \sum_{j=1}^3|\xi|^{1-j} a_{j,k}(t,x;\xi)\widehat{f}_j(\xi)
\right)
\, d\xi \Big|_{\partial\Omega},
\end{equation}
where $\simeq$ means equality up to a smooth function.
\begin{remark}
    Notice term $k=3$ does not appear in the expression above. This is because it is associated with a pseudodifferential operator, and consequently it does not spatially propagate singularities, which means that it is regarded as a smooth error when restricted to the boundary.
\end{remark}

By relabeling $k=1,2$ with $k=\pm$, respectively, the operator $\Lambda$ admits the microlocal decomposition
\[
\Lambda = \Lambda_+ + \Lambda_-,
\]
where $\Lambda_\pm : C_0^\infty(\Omega) \to C^\infty((0,\infty)\times\partial\Omega)$ are Fourier integral operators of order zero, whose canonical relations are given by the graphs of the maps
\[
(y,\xi) \mapsto \big(\pm t(y,\xi), \gamma_{y,\xi}(\pm t(y,\xi)), |\xi|, \dot{\gamma}_{y,\xi}(\pm t(y,\xi))\big).
\]
Here, $\gamma_{y,\xi}(s)$ denotes the unit speed geodesic issued from the point $y$ in the direction $\xi$, for the metric $\frac{\tau}{b}dx^2$. The function $t(y,\xi)$ represents the travel time needed for this curve to reach the boundary $\partial \Omega$, that is, it is defined by the condition $\gamma_{y,\xi}(t(y,\xi)) \in \partial \Omega$ (assuming a non-trapping condition). The signs $\pm$ correspond to forward and backward propagation along the geodesic. This decomposition reflects the propagation of singularities for the MGT operator: two wave-type branches generate bicharacteristics that intersect the boundary and are therefore detected by the measurement operator. The remaining branch, which is not considered in \eqref{eq:meas_op_FIO}, corresponds to a non-propagating mode with vertical null-bicharacteristics. The associated singularities remain spatially localized and consequently, microlocally invisible from the boundary.

\subsection{Backward parametrix}


Let $\mathcal{K}$ be a compact subset of $\Omega$ and define the quantity $T_1(\mathcal{K}, \Gamma)$ as the minimal time such that for every $(x,\xi) \in S^* \mathcal{K}$, at least one of the geodesics $\gamma_{x,\xi}$ or $\gamma_{x,-\xi}$ reaches $\Gamma$ before time $T_1(\mathcal{K}, \Gamma)$. We say that $\mathcal{K}$ is visible from $\Gamma$ if $T_1(\mathcal{K}, \Gamma)<\infty$.

For a visible set $\mathcal{K}$ from $\Gamma$, we can always find an open subset $\Gamma'\Subset\Gamma$ such that $T_1(\mathcal{K},\Gamma')<\infty$. Then, we let $\chi\in C^\infty((0,T) \times \partial\Omega)$ be a compactly supported cut-off, with support contained in $\Gamma$ and such that $\chi =1$ in $\Gamma'$. We then consider the operator $\chi \Lambda$ and an observation time $T > T_1(\mathcal{K}, \Gamma)$. 

Following the ideas presented in \cite{MR2519863}, we back-propagate singularities by solving the system

\begin{equation}
\begin{cases}\label{boundary-system}
Pv = 0 & \text{in } \Omega\times (0, T) , \\
v|_{[0,T] \times \partial \Omega} = \chi h,
\end{cases}    
\end{equation}
and then study the normal operator obtained by composing the forward and backward propagation.

Provided that $\text{supp}(f)\subset \mathcal{K}$, any singularity $\text{WF}(f)$ will be propagated to a singularity in $\text{WF}(\chi\Lambda f)\subset T^*\big(\Gamma\times(0,T)\big)$.
Assuming for a moment that $\Lambda f$ has its wave front set contained in a small conic neighborhood of some
$((y_0,t_0), (\eta_0',\sigma_0))\in T^*\big(\Gamma\times(0,T)\big)$,
 obtained by the null-bicharacteristic flow out from $(x_0,\xi_0)\in WF(f)$ ---
this can be achieved by applying a microlocal cut-off --- 
we apply geometric optics again, this time over \eqref{boundary-system}, thus with data $h = \Lambda f$ posed on $\mathbb{R} \times \partial \Omega$. This generates two types of singularities: one type propagates to the initial time with spatial projections lying along geodesics close to $\gamma_{x_0, \xi_0}$;
    the spatial projections of the second type follow geodesics near the one issued from $(y_0,(\eta'_0,-\eta_0''))\in T^*\overline{\Omega}$, with $\eta_0=(\eta_0',\eta''_0)$ such that $p_3(y_0,\sigma_0,\eta_0)=0$ and outward pointing.
    These correspond to reflected rays (or broken geodesics) traveling forward in time (towards $t=T$).
Because we have that $h = \chi \Lambda f = 0$ for $t$ close to $T$, the second type of singularity is ruled out by our microlocal construction. 

Denoting by $A:H^1(\Gamma\times(0,T))\to L^2(\Omega)$ the bounded operator mapping $h\mapsto v''(0)$, with $v$ solution to the system \eqref{boundary-system}, we want to microlocally analyze the normal operator $A\chi \Lambda$ in our case of interest $\mathbf{f} = [0,0,f]$. 
This analysis is summarized in the following lemma.

\begin{lemma}\label{principal-symbol-composition}
The composition $A\chi \Lambda: L^2(\Omega) \rightarrow L^2(\Omega)$, with $\text{supp}(\chi) \subset \Gamma$, is a pseudodifferential operator of order zero with principal symbol 
\[
-\frac{1}{2}\left[\chi\left(\tau_{y, \eta}, \gamma_{y, \eta}\left(\tau_{y, \eta}\right)\right)+\chi\left(\tau_{y,-\eta}, \gamma_{y,-\eta}\left(\tau_{y,-\eta}\right)\right)\right].
\]
\end{lemma}
\begin{proof}
The idea is to construct an oscillatory integral formulation of $A\chi \Lambda f$ using a geometric optics expansion based on the two characteristic roots $\pm \sqrt{\dfrac{b}{\tau}}|\xi|$. In the microlocal representation of $\Lambda f$, there is also a zero characteristic root, but we disregard it because the singularities it generates are stationary and therefore never interact with the boundary $\partial \Omega$.

Now, since $\mathcal{K}$ is visible from $\Gamma$, we can construct a parametrix for the solution of the boundary value problem \eqref{boundary-system} with zero Cauchy data at $t = T$ and boundary data $\chi \Lambda_+ f$, of the form

\begin{equation}\label{eq:anzatsforboundary}
 v(x,t) = (2\pi)^{-n} \int e^{i\tilde{\phi}(x,t,\xi)} B(x,t, \xi) \, |\xi|^{-2}\hat{f}(\xi) \, d\xi.   
\end{equation}
For this, by means of microlocal cut-offs and Egorov's theorem, we can always assume that $f$ has a wavefront set contained in a conic neighborhood of some $(x_0,\xi_0)\in T^*\mathcal{K}$, and consequently, the construction of the phase and amplitude is done near the associated null bicharacteristic.

The phase function $\tilde\phi$ is chosen to satisfy the same eikonal equation and the same boundary condition as the outgoing phase $\phi_+$ in the forward parametrix. Hence, by uniqueness of the local solution to this boundary value problem, we obtain
$
\tilde\phi=\phi_+.
$
The boundary parametrix is then constructed using the same phase as in the forward parametrix, while the difference between the two constructions is encoded in the amplitude $B$. 
Our goal is to determine $B$ at $t = 0$. The amplitude $B$ satisfies the same transport equation as in the construction of the forward parametrix, but now with initial data posed on $\mathbb{R} \times \partial \Omega$. Specifically, $B$ is a classical amplitude of order zero, with principal symbol $B_0$ satisfying the equation \eqref{eq: coefficientsa31}, which takes the form of a linear homogeneous ODE along the null bicharacteristics, with boundary value given by $\chi a|_{\Gamma\times(0,T)}$ with $a:=a^{(0)}_{3,1}$ defined by \eqref{eq: coefficientsa31}-\eqref{initial-conditions-for-principal-term}. The functions $B_0$ and $a$ are solutions to the same transport equation but with boundary values differing only by a multiplicative factor. For each $(x,\xi)\in T^*\mathcal{K}$, this factor is given by the evaluation of $\chi(x,t)$ at the exit point and time of the associated geodesic $\gamma_{x,\xi}$, namely $\chi_0(x,\xi):=\chi(\gamma_{x,\xi}(\tau_{x,\xi}),\tau_{x,\xi})$, then
 $B_0(x,\xi)|_{t=0}= \chi_0 a|_{t=0}=-\dfrac{\chi_0 \tau}{2b}$.
Evaluating $\partial_t^2v$ at $t = 0$ gives a representation of $A\chi \Lambda_+$ as a pseudodifferential operator with principal symbol 
\[
|\xi|^{-2}|\partial_t\phi_+|^2B_0\big|_{t=0} = \frac{b}{\tau}B_0(x,0,\eta) = -\frac{1}{2}\chi_0(x,\xi),
\]
which follows by differentiating twice \eqref{eq:anzatsforboundary} and recalling \eqref{eikonal-wave}. 

The same reasoning applies to the term corresponding to the negative characteristic root. Consequently, taking into account both contributions from singularities traveling along $\gamma_{y_0, \eta_0}$ and $\gamma_{y_0, -\eta_0}$, we find that $A\chi\Lambda$ is a zero order pseudodifferential operator with principal symbol given by
\[
-\frac{1}{2} \left[ \chi(\tau_{y_0, \eta_0}, \gamma_{y_0, \eta_0}(\tau_{y_0, \eta_0})) + \chi(\tau_{y_0, -\eta_0}, \gamma_{y_0, -\eta_0}(\tau_{y_0, -\eta_0})) \right].
\]
\end{proof}
The following theorem establishes the Fredholm property, ellipticity, and microlocal stability estimate for recovering $f$ from partial boundary data.
\begin{theorem}\label{stability-thm}
Assume that $\partial\Omega$ is strictly convex with respect to the the metric $\frac{\tau}{b}dx^2$, and that $T_1(\mathcal{K}, \Gamma) < T < \infty$. 
Then the following statements hold:
\begin{enumerate}

    \item The operator $A\chi \Lambda : L^2(\mathcal{K}) \to L^2(\Omega)$ is an elliptic pseudodifferential operator of order zero. Here, functions in $L^2(\mathcal{K})$ are extender by zero to all $\Omega$.
    \item The operator $A\chi \Lambda : L^2(\mathcal{K}) \to L^2(\Omega)$ is Fredholm.

    \item There exists a constant $C > 0$ and a smoothing operator $R$, such that
\begin{equation}
\|f\|_{L^2(\Omega)} \leq C \left(\|\Lambda f\|_{H^1(\Sigma)} + \|Rf\|_{L^2(\Omega)}\right).
\end{equation}
\end{enumerate}
\end{theorem}
\begin{proof}

\textit{(1) Ellipticity:} The principal symbol of \( A\chi \Lambda \), computed in the Lemma \ref{principal-symbol-composition}, is homogeneous of order zero and is nonzero since $T_1(\mathcal{K},\Gamma)< T$, implying the operator is elliptic and of order zero.

\textit{(2) Fredholm property:} This is a direct implication of the ellipticity property.

\textit{(3) Stability inequality:} Since $A\chi \Lambda$ is elliptic, there exists a $\Psi$DO of order zero $Q$, such that $QA\chi \Lambda = I + R$ where $R$ is a smoothing operator. We then get
\[
\|f\|_{L^2} \leq C\left(\|A\chi\Lambda  f\|_{L^2} + \|R f\|_{L^2}\right).
\]
Since $A$ is bounded from $H^1(\Sigma)$ to $L^2(\Omega)$, we have
\[
\|f\|_{L^2} \leq C'\big(\|\Lambda f\|_{H^1} + \|Rf\|_{L^2}\big).
\]

\end{proof}

We remark that Theorem \ref{stability-thm} is formulated under partial observation. 
When global boundary data is available --- and under appropriate geometric conditions --- the measurement operator is injective (thanks to Corollary \ref{cor:uniqueness}) and the residual term can be removed from the previous inequality. In this case, the estimate reduces to the global stability result as stated in Theorem \ref{main-thm}.
\begin{corollary}[Stability for the inverse problem with complete data]\label{cor:stability-thm}
    Let $\Omega$ be a bounded open set with a smooth and strictly convex boundary for the metric $g=\frac{\tau}{b}dx^2$, and let $\mathcal{K}\subset\Omega$ compact and such that $T_1(\mathcal{K},\partial\Omega)<\infty$.  Assume that  $x^n=\operatorname{dist}_g(x,\partial\Omega)$ is smooth with strictly convex level sets $\{x^n=s\}$ for $0\le s\le T$, with $T>T_1(\mathcal{K},\partial\Omega)$. Then, there is $C>0$ such that
    \[
    \|f\|_{L^2(\mathcal{K})} \leq C\|\Lambda f\|_{H^1(\partial\Omega\times (0,T))}\qquad\forall f\in L^2(\mathcal{K}). 
    \]
\end{corollary}

\section{Acknoledgment}

This work was partially supported by the Agencia Nacional de Investigaci\'on y Desarrollo (ANID), Grant FONDECYT  N$^\circ$1251207.

\section{Appendix}\label{APPENDIX}
For the sake of completeness, we will state and prove the Local Carleman estimate that was used in \S\ref{STEP-1}. Before stating the result, we present the notation. For scalar functions, we denote $\partial_{x^i} \varphi=\varphi_i, \varphi_x=\left(\varphi_1, \ldots, \varphi_n\right), \varphi^{\prime}=$ $\left(\varphi_x, \varphi_t\right) ; \varphi_{x x}=\left(\varphi_{j k}\right) ;\langle\xi, \eta\rangle_g=g^{i j} \xi_i \eta_j$ is the inner product induced by the metric $g$ in the cotangent bundle, and $|\xi|_g^2=g^{i j} \xi_i \xi_j$ is the corresponding norm; $\left(\nabla_g \varphi\right)=\left(g^{i j} \varphi_i\right)_j$ the gradient of $\varphi$ for the metric $g$; for a vector-valued function $f=(f^1,\dots,f^n)$,  $\operatorname{div}_x\left(f\right)=\partial_{x^i} f^i$ the divergence operator (for the euclidean metric). 
\begin{lemma}\label{lemma-local-carleman-benjamin}
There exists $\tau_0>0$ and $C>0$, depending on $Q_0$ and $\Gamma_0$, such that for all $u\in C^2(\overline{Q}_0)$ satisfying $u = \frac{\partial u}{\partial\nu} = 0$ on $\partial Q_0\cap\partial\Omega$, and one of the following cases for the initial conditions:
\[
u(x,0)=0\quad\text{or}\quad u_t(x,0)=0\quad\text{or}\quad u_t(x,0)=\rho(x)u(x,0),
\]
for some $\rho\in C^2(\bar{\Omega})$; then, the next estimate holds for every $\tau\geq\tau_0$:
\begin{equation}\label{proof-local-carleman-wave-equation}
\begin{aligned}
&\int_{Q_0}(\tau^3|u|^2 + \tau|u'|^2)e^{2\tau\varphi}dxdt + \int_{\Gamma_0}\big(\tau\langle Au',u'\rangle + \tau\langle b,u'\rangle u + \tau^3d|u|^2\big)dS\\
&\hspace{10em} \leq C\int_{Q_0}|P_0u|^2e^{2\tau\varphi}dxdt.
\end{aligned}
\end{equation}
Here $dS$ denotes the surface measure on $\Gamma_0$, and the matrix-function $A(x,t)$, the vector-function $b(x,t)$, and the scalar-function $d(x,t)$, are bounded and piecewise smooth functions depending on $\varphi$ and $Q_0$.

\end{lemma}

\begin{proof}
The goal is to derive the Carleman estimate by bounding from below the integral $\int_{Q_0} |P_\tau v|^2 \, dx \, dt$.
We compute the derivatives of $\tilde{u}$:
\[
\tilde{u}_t = (v_t - \tau \varphi_t v) e^{-\tau \varphi}, \quad \tilde{u}_i = (v_i - \tau \varphi_i v) e^{-\tau \varphi},
\]
\[
\tilde{u}_{tt} = (v_{tt} - 2 \tau \varphi_t v_t + \tau^2 |\varphi_t|^2 v - \tau \varphi_{tt}) e^{-\tau \varphi}.
\]
For the spatial derivatives, we have:
\[
\begin{aligned}
\partial_j \left( g^{ij} \partial_i \tilde{u} \right) &= \partial_j \left( g^{ij} v_i e^{-\tau \varphi} \right) - \tau \partial_j \left( g^{ij} \varphi_i v e^{-\tau \varphi} \right)\\
&= \partial_j \left( g^{ij} v_i \right) e^{-\tau \varphi} - 2 \tau \langle \varphi_x, v_x \rangle_g e^{-\tau \varphi} - \tau \left( \partial_j \left( g^{ij} \varphi_i \right) \right) v e^{-\tau \varphi} + \tau^2 |\varphi_x|_g^2 v e^{-\tau \varphi}.
\end{aligned}
\]
Thus, we find:
\[
P_\tau v = P_{+} v + P_{-} v,
\]
where
\[
P_{+} v = v_{tt} - \partial_j (g^{ij} \partial_i v) + \tau^2 \Phi v, \quad \Phi = \varphi_t^2 - |\varphi_x|_g^2,
\]
\[
P_{-} v = 2 \tau \left( \langle \varphi_x, v_x \rangle_g - \varphi_t v_t \right) + \tau \Psi v, \quad \Psi = \partial_j (g^{ij} \varphi_i) - \varphi_{tt}.
\]
Since 
\[
|P_\tau v|^2 = |P_{+} v|^2 + 2 (P_{+} v)(P_{-} v) + |P_{-} v|^2 \geq |P_{+} v|^2 + 2 (P_{+} v)(P_{-} v), 
\]
we need to estimate the two terms on the right-hand side from below. Let’s first compute the derivatives of $\varphi$. Below, $\delta_{ij}$ stands for the Kronecker delta. We have:
\begin{equation}\label{PROOF-CARLEMAN-WAVE-values-of-derivative}
\begin{aligned}
\varphi_t = -2 a t, \quad \varphi_{tt} &= -2 a, \quad \varphi_i = \delta_{in} \left( -2 (R - x^n) \right), \quad \varphi_{ij} = 2 \delta_{in} \delta_{jn},\\
&\Phi = 4 (a^2 t^2 - (R - x^n)^2), \quad \Psi = 2 + 2 \alpha.
\end{aligned}
\end{equation}
With this, we can now estimate both terms.\\

\noindent$\mathbf{\textbf{Estimate for} \, \left|P_{+} v\right|^2.}$ We see that in the set $\{\varphi>0\}$ we have
\[
\left(R-x^n\right)^2-a t^2>r^2.
\]
Thus, using \eqref{PROOF-CARLEMAN-WAVE-values-of-derivative} and choosing $a \in(0,1)$, we get
\[
-\Phi \geq 4\left(\left(R-x^n\right)^2-a t^2\right)>4 r^2.
\]
Moreover, for any $\tilde{\gamma} > 0$, since $\frac{1}{4 r^2} > \frac{1}{-\Phi}$, if we multiply by $\tilde{\gamma} / \tau_0$ it is clear that for any $r > 0$ we can find $\tau_0 > 0$ such that for all $\tau \geq \tau_0$,
\[
0 < \frac{-\tilde{\gamma}}{\Phi \tau} < \frac{-\tilde{\gamma}}{\Phi \tau_0} \leq \frac{\tilde{\gamma}}{4 r^2 \tau_0} < 1.
\]
Then,
\[
\begin{aligned}
\left|P_{+} v\right|^2 &\geq \frac{-\tilde{\gamma}}{\Phi \tau} \left|P_{+} v\right|^2 = \frac{-\tilde{\gamma}}{\Phi \tau} \left|v_{tt} - \partial_j \left( g^{ij} \partial_i v \right) + \tau^2 \Phi v\right|^2 \\
&= \frac{-\tilde{\gamma}}{\Phi \tau} \left|v_{tt} - \partial_j \left( g^{ij} \partial_i v \right)\right|^2 - \tilde{\gamma} \tau^3 \Phi |v|^2 + 2 \tau \tilde{\gamma} \left( \partial_j \left( g^{ij} \partial_i v \right) - v_{tt} \right) v.
\end{aligned}
\]
Because,
\[
\left( \partial_j \left( g^{ij} \partial_i v \right) - v_{tt} \right) v = \operatorname{div}_x \left( v \left( \nabla_g v \right) \right) - \left|v_x\right|_g^2 - \partial_t \left( v v_t \right) + \left|v_t\right|^2,
\]
we end up with the estimate
\[
\begin{aligned}
\left|P_{+} v\right|^2 &\geq \frac{-\tilde{\gamma}}{\Phi \tau} \left|v_{tt} - \partial_j \left( g^{ij} \partial_i v \right)\right|^2 - \tilde{\gamma} \tau^3 \Phi |v|^2 + 2 \tau \tilde{\gamma} \left( \left|v_t\right|^2 - \left|v_x\right|_g^2 \right) \\
&\quad + 2 \tau \tilde{\gamma} \left( \operatorname{div}_x \left( v \left( \nabla_g v \right) \right) - \partial_t \left( v v_t \right) \right) \\
&\geq 2 \tau \tilde{\gamma} \left( \left|v_t\right|^2 - \left|v_x\right|_g^2 \right) + 2 \tau \tilde{\gamma} \left( \operatorname{div}_x \left( v \left( \nabla_g v \right) \right) - \partial_t \left( v v_t \right) \right).
\end{aligned}
\]

\noindent $\mathbf{\textbf{Estimate of} \, \, 2\left(P_{+} v\right)\left(P_{-} v\right)}$. We split the computations into 3 parts:

\[
\begin{aligned}
2\left(P_{+} v\right)\left(P_{-} v\right)&= 2\left(v_{t t}-\partial_j\left(g^{i j} \partial_i v\right)+\tau^2 \Phi v\right)\left(2 \tau\left(\left\langle\varphi_x, v_x\right\rangle_g-\varphi_t v_t\right)+\tau \Psi v\right) \\
& =I_1+I_2+I_3 .
\end{aligned}
\]
Expanding these terms we obtain the following. For $I_1$ we have
\[
\begin{aligned}
I_1= & 4 \tau\left\langle\varphi_x, v_x\right\rangle_g v_{t t}-4 \tau \varphi_t v_t v_{t t}+2 \tau \Psi v v_{t t} \\
= & \partial_t\left(4 \tau\left\langle\varphi_x, v_x\right\rangle_g v_t-2 \tau \varphi_t\left|v_t\right|^2+2 \tau \Psi v v_t\right) \\
-4&  \tau\left\langle\varphi_x, v_{t x}\right\rangle_g v_t+2 \tau \varphi_{t t}\left|v_t\right|^2-2 \tau \Psi_t v v_t-2 \tau \Psi\left|v_t\right|^2 \\
= & \partial_t\left(4 \tau\left\langle\varphi_x, v_x\right\rangle_g v_t-2 \tau \varphi_t\left|v_t\right|^2+2 \tau \Psi v v_t\right)\\
+& \operatorname{div}_x\left(-2 \tau\left|v_t\right|^2\left(\nabla_g \varphi\right)\right)+2\tau\operatorname{div}_x(\nabla_g\varphi)|v_t|^2+2 \tau \varphi_{t t}\left|v_t\right|^2- 2\tau\Psi_tvv_t -2 \tau \Psi\left|v_t\right|^2\\
= & \partial_t\left(4 \tau\left\langle\varphi_x, v_x\right\rangle_g v_t-2 \tau \varphi_t\left|v_t\right|^2+2 \tau \Psi v v_t\right)\\
+& \operatorname{div}_x\left(-2 \tau\left|v_t\right|^2\left(\nabla_g \varphi\right)\right)+4 \tau \varphi_{t t}\left|v_t\right|^2.
\end{aligned}
\]
For the term $I_2$ we get
\[
\begin{aligned} 
I_2  = &-4 \tau\left(\partial_j\left(g^{i j} \partial_i v\right)\right)\left\langle\varphi_x, v_x\right\rangle_g+4 \tau\left(\partial_j\left(g^{i j} \partial_i v\right)\right) \varphi_t v_t-2 \tau \Psi\left(\partial_j\left(g^{i j} \partial_i v\right)\right) v \\  = &\operatorname{div}_x\left(-4 \tau\left\langle\varphi_x, v_x\right\rangle_g\left(\nabla_g v\right)\right)+4 \tau\left\langle\nabla_g v, \nabla\left\langle\varphi_x, v_x\right\rangle_g\right\rangle\\
+ &\operatorname{div}_x\left(4 \tau \varphi_t v_t \nabla_g v\right)-4 \tau v_t\left\langle\nabla_g v, \nabla \varphi_t\right\langle - 4 \tau \varphi_t\left\langle\nabla_g v, \nabla v_t\right\rangle\\
+&\operatorname{div}_x\left(-2 \tau \Psi v\left(\nabla_g v\right)\right)+2 \tau|\nabla v|_g^2 \Psi+2 \tau v\left\langle\nabla_g v,\nabla\Psi\right\rangle\\
= & \operatorname{div}_x\left(-4 \tau\left\langle\varphi_x, v_x\right\rangle_g\left(\nabla_g v\right)\right)+4 \tau\left\langle\nabla_g v, \nabla\left\langle\varphi_x, v_x\right\rangle_g\right\rangle \\  +&\operatorname{div}_x\left(4 \tau \varphi_t v_t \nabla_g v\right)+\partial_t\left(-2 \tau\left|v_x\right|_g^2 \varphi_t\right)+2 \tau \varphi_{t t}\left|v_x\right|_g^2 \\
+ &\operatorname{div}_x\left(-2 \tau \Psi v\left(\nabla_g v\right)\right)+2 \tau\left|v_x\right|_g^2 \Psi,
\end{aligned}
\]
thus, expanding 
\begin{align*}
\left\langle\nabla_g v, \nabla\left\langle\varphi_x, v_x\right\rangle_g\right\rangle &=\left(g^{k l} v_k\right) \cdot \partial_l\left(g^{i j} \varphi_i v_j\right)\\
&= \left(g^{k l} v_k\right)\left(g_{\ell}^{i j} \varphi_i v_j\right) +\left(g^{k l} v_k\right)\left(g^{i j} v_j\right) \varphi_{i l}+\left(g^{k l} v_k\right)\left(g^{i j} \varphi_i\right) v_{j l}\\
&= \varphi_{x x}\left(\nabla_g v, \nabla_g v\right)+\frac{1}{2}\left\{\operatorname{div}_x\left(\left|v_x\right|_g^2\left(\nabla_g \varphi\right)\right)-\left(g_j^{k l} v_k v_l\right)\left(\nabla_g \varphi\right)\right. \\
&\left.-\left|v_x\right|_g^2\left(g_j^{i j} \varphi_i\right) -\left|v_x\right|_g^2\left(g^{i j} \varphi_{i j}\right)\right\},
\end{align*}
where there are two summands vanishing because $\varphi_i \neq 0$ only if $i=n$ and $\partial_k g^{n j}=0$ for all $j, k$. Then, we obtain
\[
\begin{aligned}
I_2= & \operatorname{div}_x\left(-4 \tau\left\langle\varphi_x, v_x\right\rangle_g\left(\nabla_g v\right)+2 \tau\left|v_x\right|_g^2\left(\nabla_g \varphi\right)\right)+4 \tau \varphi_{x x}\left(\nabla_g v, \nabla_g v\right) \\
& -2 \tau\left\langle g_j^{k l} v_k v_l, \nabla_g \varphi\right\rangle-2 \tau\left|v_x\right|_g^2\left(g^{i j} \varphi_{i j}\right) \\
& +\operatorname{div}_x\left(4 \tau \varphi_t v_t \nabla_g v\right)+\partial_t\left(-2 \tau\left|v_x\right|_g^2 \varphi_t\right)+2 \tau \varphi_{t t}\left|v_x\right|_g^2 \\
& +\operatorname{div}_x\left(-2 \tau \Psi v\left(\nabla_g v\right)\right)+2 \tau\left|v_x\right|_g^2 \Psi.
\end{aligned}
\]
Finally, the term $I_3$ is such that
\[
\begin{aligned}
I_3&=4 \tau^3 \Phi v\left\langle\varphi_x, v_x\right\rangle_g-4 \tau^3 \Phi \varphi_t v v_t+2 \tau^3 \Phi \Psi|v|^2\\
&=\operatorname{div}_x\left(2 \tau^3 \Phi\left(\nabla_g \varphi\right)|v|^2\right)-2 \tau^3\left\langle\Phi_x, \nabla_g \varphi\right\rangle|v|^2-2 \tau^3 \Phi\left(\partial_j\left(g^{i j}\varphi_i\right)\right)|v|^2\\
&+\partial_t\left(-2 \tau^3 \Phi \varphi_t|v|^2\right)+2 \tau^3 \Phi_t \varphi_t|v|^2+2 \tau^3 \Phi \varphi_{t t}|v|^2 +2 \tau^3 \Phi \Psi|v|^2
\end{aligned}
\]
Putting all the above together, we get
\[
\begin{aligned}
2\left(P_{+} v\right)\left(P_{-} v\right)= & 2 \tau\left(\varphi_{t t}-\Psi\right)\left|v_t\right|^2 \\
& +4 \tau \varphi_{x x}\left(\nabla_g v, \nabla_g v\right)-2 \tau\left\langle g_j^{k l} v_k v_l, \nabla_g \varphi\right\rangle \\
& +2 \tau^3\left(\Phi_t \varphi_t-\left\langle\Phi_x, \nabla_g \varphi\right\rangle\right)|v|^2 \\
& +\operatorname{div}_x(\tilde{Y})+\partial_t \tilde{Z},
\end{aligned}
\]
where
\begin{equation}\label{def:YZ_tilde}
\begin{aligned}
\tilde{Y}= & -2 \tau\left|v_t\right|^2\left(\nabla_g \varphi\right)-4 \tau\left\langle\varphi_x, v_x\right\rangle_g\left(\nabla_g v\right)+2 \tau\left|v_x\right|_g^2\left(\nabla_g \varphi\right)+4 \tau \varphi_t v_t\left(\nabla_g v\right) \\
& -2 \tau \Psi v\left(\nabla_g v\right)+2 \tau^3 \Phi|v|^2\left(\nabla_g \varphi\right) \\
\tilde{Z} & =4 \tau\left\langle\varphi_x, v_x\right\rangle_g v_t-2 \tau \varphi_t\left|v_t\right|^2+2 \tau \Psi v v_t-2 \tau\left|v_x\right|_g^2 \varphi_t-2 \tau^3 \Phi \varphi_t|v|^2.
\end{aligned}
\end{equation}
Using \eqref{PROOF-CARLEMAN-WAVE-values-of-derivative}, we see that 
\[
\begin{aligned}
 2\left(P_{+} v\right)(P-v)=&-8\alpha\tau\left|v_t\right|^2 \\
& \quad+8 \tau\left|v_n\right|^2+8\left(R-x^n\right) \tau\left(\frac{1}{2} \partial_n g^{k l} v_k v_l\right)\\
& \quad+32 \tau^3\left(\left(R-x^n\right)^2-\alpha^3 t\right)|v|^2  +\operatorname{div}_x(\tilde{Y})+\partial_t \tilde{Z}.
\end{aligned}
\]

Recalling the strict convexity of the level surfaces $\left\{x^n=s\right\}$ in \eqref{positive-definiteness-inequality}, by choosing $r=R-\delta$ in the definition of the weight function, we get that $x^n \in[0, \delta)$ inside $Q_0=Q \cap\{\varphi>0\}$. Therefore, by taking $\kappa=\min _{s \in[0, \delta]} \kappa_s$, we have that
\[
\frac{1}{2} \partial_n g^{k l} v_k v_l \geq \kappa\left|v_x\right|_{\tilde{g}}^2 \quad \text { in } Q_0,
\]
thus, in $Q_0$ we get
\[
\begin{aligned}
\left|P_\tau v\right|^2 \geq & 2 \tau(\tilde{\gamma}-4 \alpha)\left|v_t\right|^2 \\
& +2 \tau(4-\tilde{\gamma})\left|v_n\right|^2+2 \tau\left(2\left(R-x^n\right) \kappa-\tilde{\gamma}\right)\left|v_x\right|_{\tilde{g}}^2 \\
& +32 \tau^3 r^2|v|^2+\operatorname{div}_x(Y)+\partial_t Z,
\end{aligned}
\]
with
\begin{equation}\label{def:YZ}
Y=\tilde{Y}+2 \tau \tilde{\gamma} v\left(\nabla_g v\right), \quad Z=\tilde{Z}+2 \tau \tilde{\gamma} v v_t.
\end{equation}
We then choose $\tilde{\gamma}$ and $a \in(0,1)$ such that
\[
4>\tilde{\gamma}>4 a,
\]
and $R$ and $\delta$ so that in $Q_0$,
\[
(R-x^{n})-\frac{\tilde\gamma}{4\kappa}>(R-\delta)-\frac{\tilde\gamma}{4\kappa}>0.
\]
In summary, and after integrating over $Q_0^y$, we obtain the local estimate:
\begin{equation*}
 C \tau \int_{Q_0^y}\left(\left|v_t\right|^2+\left|v_x\right|_g^2\right)+C \tau^3 \int_{Q_0^y}|v|^2 d x d t \leq \int_{Q_0^y}\left|P_\tau v\right|^2 d x d t-\int_{Q_0^y}\left(\operatorname{div}_x Y+\partial_t Z\right) d x d t.  \end{equation*}

If the initial conditions are $u(x,0)=0$ or $u_t(x,0)=0$, then $Z(x,0)=0$, therefore, applying the Gauss-Ostrogradskii formula (divergence theorem) along with the fact that $v|_{\Sigma}=\frac{\partial v}{\partial\nu}\big|_\Sigma = 0$, and returning to the original function $u$, we get a local version (near $y$) of the desired estimate with $u$ replaced by $\tilde{u}$ and the sets $Q_0$, $\Gamma_0$ replaced by $Q^y_0$ and $\Gamma^y_0$ respectively.

We patch together a finite number of these local estimates to obtain \eqref{proof-local-carleman-wave-equation} in the cases $u(x,0)=0$ or $u_t(x,0)=0$. Indeed, by compactness of the boundary, there is a finite collection of boundary points $y_i\in \partial\Omega$ with respective boundary neighborhoods $U_i$, such that $\{U_i\}$ is a finite covering of $\partial\Omega$. Each set has local boundary coordinates and a local Carleman estimate holds. Denoting $u_i=\chi_i u$, we have:
$$
\begin{aligned}
\tau^3\int_{Q_0}e^{2\tau\varphi}|u|^2dxdt + \tau\int_{Q_0}e^{2\tau\varphi}(|u_t|^2 + |u_x|^2_g)dxdt&\\
&\hspace{-20em} =\sum_i\left(\tau^3\int_{Q_0}\chi_i^2e^{2\tau\varphi}|u|^2dxdt + \tau\int_{Q_0}\chi_i^2e^{2\tau\varphi}(|u_t|^2 + |u_x|^2_g)dxdt\right) \\
&\hspace{-20em} \leq \sum_i\left(\tau^3\int_{Q^{y_i}_0}e^{2\tau\varphi}|u_i|^2dxdt + \tau\int_{Q^{y_i}_0}e^{2\tau\varphi}(|(u_i)_t|^2 + |(u_i)_x|^2_g)dxdt\right) \\
&\hspace{-10em} + C\tau\int_{Q_0}e^{2\tau\varphi}|u|^2dxdt\\
&\hspace{-20em} \leq C\sum_i\left(\int_{Q^{y_i}_0}e^{2\tau\varphi}|P_0 u_i|^2dxdt+ \int_{\Gamma^{y_i}_0}\big(\tau\langle A_iu'_i,u'_i\rangle + \tau\langle b_i,u'_i\rangle u_i + \tau^3d_i|u_i|^2\big)dS\right)\\
&\hspace{-10em} + C\tau\int_{Q_0}e^{2\tau\varphi}|u|^2dxdt\\
&\hspace{-20em} \leq C\sum_i\int_{Q^{y_i}_0}e^{2\tau\varphi}(\chi_i^2|P_0 u|^2 + |[P_0,\chi_i]u|^2)dxdt\\
&\hspace{-15em} + C\int_{\Gamma_0}\big(\tau\langle Au',u'\rangle + \tau\langle b,u'\rangle u + \tau^3d|u|^2\big)dS+ C\tau\int_{Q_0}e^{2\tau\varphi}|u|^2dxdt\\
\end{aligned}
$$
where $[P_0,\chi_i]$ are differential operators of order 1. We can then absorb the interior integrals with lower order derivatives of $u$ with the left hand side and get
\begin{equation}\label{car3}
\begin{aligned}
\tau^3\int_{Q(0)}e^{2\tau\varphi}|u|^2dxdt + \tau\int_{Q(0)}e^{2\tau\varphi}(|u_t|^2 + |u_x|^2_g)dxdt&\\
&\hspace{-21em} \leq C\int_{Q(0)}e^{2\tau\varphi}|P_0 u|^2dxdt+ C\int_{\Gamma(0)}\big(\tau\langle Au',u'\rangle +\tau \langle b,u'\rangle u + \tau^3d|u|^2\big)dS.
\end{aligned}
\end{equation}

It remains to prove the estimate for initial conditions of the form
\[
u_t(x,0)=\rho(x)u(x,0).
\]
The estimate will follow directly from the previous case after applying the next change of variables:
\[
w(x,t) = e^{-t\rho}u(x,y).
\]
This new function has initial conditions
\[
w(x,0)=u(x,0)\quad\text{and}\quad w_t(x,0)=u_t(x,0)-\rho(x)u(x,0) = 0.
\]
Applying the wave operator to $u=e^{t\rho}w$ we obtain
\[
e^{-t\rho}pu = pw + 2\rho w_t - 2tg^{ij}(\partial_i\rho)\partial_jw + (\rho^2 - t\partial_jg^{ij}\partial_i\rho - t^2|\nabla\rho|^2_g)w,
\]
thus, the Carleman estimate for $w$ gives
\[
\begin{aligned}
\tau^3\int_{Q(0)}e^{2\tau\varphi}|w|^2dxdt + \tau\int_{Q(0)}e^{2\tau\varphi}(|w_t|^2 + |w_x|^2_g)dxdt&\\
&\hspace{-21em} \leq C\int_{Q(0)}e^{2\tau\varphi}|p w|^2dxdt+ C\int_{\Gamma(0)}\big(\tau\langle Aw',w'\rangle +\tau \langle b,w'\rangle w + \tau^3d|w|^2\big)dS,
\end{aligned}
\]
which, after absorbing some terms with the left hand side (for $\tau$ large enough, depending on $\rho$), gives
\[
\begin{aligned}
\tau^3\int_{Q(0)}e^{2\tau\varphi}|w|^2dxdt + \tau\int_{Q(0)}e^{2\tau\varphi}(|w_t|^2 + |w_x|^2_g)dxdt&\\
&\hspace{-21em} \leq C\int_{Q(0)}e^{2\tau\varphi}|e^{-t\rho}pu|^2dxdt+ C\int_{\Gamma(0)}\big(\tau\langle Aw',w'\rangle +\tau \langle b,w'\rangle w + \tau^3d|w|^2\big)dS,
\end{aligned}
\]
for another constant $C>0$. We then go back to $u$ using that
\[
w = e^{-t\rho}u,\quad w_t = -\rho e^{-t\rho}u + e^{-t\rho}u_t,\quad w_i = -\rho_ie^{-t\rho}u + e^{-t\rho}u_i,
\]
and that $e^{-t\rho}$ can be bounded in $\overline{Q(0)}$ from above and below by positive constants. Estimating from above and below, we deduce the desired estimate for $u$.

\end{proof}

\bibliographystyle{acm}
\bibliography{bibliography}

@article {BENJAMIN-ACOSTA-THERMO,
    AUTHOR = {Acosta, Sebasti\'an and Palacios, Benjam\'in},
     TITLE = {Thermoacoustic tomography for an integro-differential wave
              equation modeling attenuation},
   JOURNAL = {J. Differential Equations},
  FJOURNAL = {Journal of Differential Equations},
    VOLUME = {264},
      YEAR = {2018},
    NUMBER = {3},
     PAGES = {1984--2010},
      ISSN = {0022-0396,1090-2732},
   MRCLASS = {35R30 (35R09)},
  MRNUMBER = {3721418},
MRREVIEWER = {Dmitry\ E.\ Pelinovsky},
       DOI = {10.1016/j.jde.2017.10.012},
       URL = {https://doi.org/10.1016/j.jde.2017.10.012},
}

@article {TATARU,
    AUTHOR = {Tataru, Daniel},
     TITLE = {Unique continuation for operators with partially analytic
              coefficients},
   JOURNAL = {J. Math. Pures Appl. (9)},
  FJOURNAL = {Journal de Math\'ematiques Pures et Appliqu\'ees. Neuvi\`eme
              S\'erie},
    VOLUME = {78},
      YEAR = {1999},
    NUMBER = {5},
     PAGES = {505--521},
      ISSN = {0021-7824},
   MRCLASS = {35A05 (35B60)},
  MRNUMBER = {1697040},
MRREVIEWER = {Masafumi\ Yoshino},
       DOI = {10.1016/S0021-7824(99)00016-1},
       URL = {https://doi.org/10.1016/S0021-7824(99)00016-1},
}

@incollection {HORMANDER-PRINCIPALLY-NORMAL,
    AUTHOR = {H\"ormander, Lars},
     TITLE = {On the uniqueness of the {C}auchy problem under partial
              analyticity assumptions},
 BOOKTITLE = {Geometrical optics and related topics ({C}ortona, 1996)},
    SERIES = {Progr. Nonlinear Differential Equations Appl.},
    VOLUME = {32},
     PAGES = {179--219},
 PUBLISHER = {Birkh\"auser Boston, Boston, MA},
      YEAR = {1997},
      ISBN = {0-8176-3958-6},
   MRCLASS = {35L10 (35A07)},
  MRNUMBER = {2033496},
}

@article {UHLMANN,
    AUTHOR = {Bukhgeim, A. L. and Dyatlov, G. V. and Uhlmann, G.},
     TITLE = {Unique continuation for hyperbolic equations with memory},
   JOURNAL = {J. Inverse Ill-Posed Probl.},
  FJOURNAL = {Journal of Inverse and Ill-Posed Problems},
    VOLUME = {15},
      YEAR = {2007},
    NUMBER = {6},
     PAGES = {587--598},
      ISSN = {0928-0219,1569-3945},
   MRCLASS = {35L20 (35B45 35B60 49J20 93B05)},
  MRNUMBER = {2367866},
       DOI = {10.1515/jiip.2007.032},
       URL = {https://doi.org/10.1515/jiip.2007.032},
}

@book{hormander1994,
  author    = {Lars H\"ormander},
  title     = {The Analysis of Linear Partial Differential Operators. IV},
  volume    = {275},
  series    = {Grundlehren der Mathematischen Wissenschaften},
  publisher = {Springer-Verlag},
  address   = {Berlin},
  year      = {1994},
  note      = {Fourier integral operators, Corrected reprint of the 1985 original}
}

@article {MR2519863,
    AUTHOR = {Stefanov, Plamen and Uhlmann, Gunther},
     TITLE = {Thermoacoustic tomography with variable sound speed},
   JOURNAL = {Inverse Problems},
  FJOURNAL = {Inverse Problems. An International Journal on the Theory and
              Practice of Inverse Problems, Inverse Methods and Computerized
              Inversion of Data},
    VOLUME = {25},
      YEAR = {2009},
    NUMBER = {7},
     PAGES = {075011, 16},
      ISSN = {0266-5611,1361-6420},
   MRCLASS = {35R30 (76Q05)},
  MRNUMBER = {2519863},
MRREVIEWER = {Aleksander\ Denisiuk},
       DOI = {10.1088/0266-5611/25/7/075011},
       URL = {https://doi.org/10.1088/0266-5611/25/7/075011},
}

@article{StefanovUhlmann2011,
  author = {P. Stefanov and G. Uhlmann},
  title = {Thermoacoustic Tomography Arising in Brain Imaging},
  journal = {Inverse Problems},
  volume = {27},
  year = {2011},
}

@article{Palacios2022,
  author = {B. Palacios},
  title = {Photoacoustic Tomography in Attenuating Media with Partial Data},
  journal = {Inverse Problems and Imaging},
  volume = {16},
  year = {2022},
}

@article{nicaise2020wellposedness,
  title={Wellposedness and long time behavior for a general class of Moore-Gibson-Thompson equations},
  author={Nicaise, Serge and Bounadja, Hizia},
  journal={Control and Cybernetics},
  volume={49},
  year={2020}
}

@article{kuznetsov1971equation,
  title={Equation of nonlinear acoustics},
  author={Kuznetsov, VP},
  journal={Sov. Phys. Acoust.},
  volume={16},
  pages={467--470},
  year={1971}
}

@article{westervelt1963parametric,
  title={Parametric acoustic array},
  author={Westervelt, Peter J},
  journal={The Journal of the acoustical society of America},
  volume={35},
  number={4},
  pages={535--537},
  year={1963},
  publisher={Acoustical Society of America}
}

@book{hamilton2024nonlinear,
  title={Nonlinear acoustics},
  author={Hamilton, Mark F and Blackstock, David T},
  year={2024},
  publisher={Springer Nature}
}

@article{crighton1979model,
  title={Model equations of nonlinear acoustics},
  author={Crighton, David G},
  journal={Annual Review of Fluid Mechanics},
  volume={11},
  pages={11--33},
  year={1979}
}

@article{coulouvrat1992equations,
  title={On the equations of non linear acoustics},
  author={Coulouvrat, F},
  journal={Journal d'acoustique (Les Ulis)},
  volume={5},
  number={4},
  pages={321--359},
  year={1992}
}

@incollection{greenleaf1974algebraic,
  title={Algebraic reconstruction of spatial distributions of acoustic absorption within tissue from their two-dimensional acoustic projections},
  author={Greenleaf, James F and Johnson, Sahvanna A and Lee, So L and Hermant, GT and Woo, EH},
  booktitle={Acoustical Holography: Volume 5},
  pages={591--603},
  year={1974},
  publisher={Springer}
}

@article{dines1979ultrasonic,
  title={Ultrasonic attenuation tomography of soft tissues},
  author={Dines, KA and Kak, AC},
  journal={Ultrasonic Imaging},
  volume={1},
  number={1},
  pages={16--33},
  year={1979},
  publisher={Elsevier}
}

@article{kaltenbacher2011wellposedness,
  title={Wellposedness and exponential decay rates for the Moore-Gibson-Thompson equation arising in high intensity ultrasound},
  author={Kaltenbacher, Barbara and Lasiecka, Irena and Marchand, Richard},
  journal={Control and Cybernetics},
  volume={40},
  number={4},
  pages={971--988},
  year={2011},
  publisher={Polska Akademia Nauk. Instytut Bada{\'n} Systemowych PAN}
}

@article{bjornum1986,
  author = {Bjørnum, E.},
  title = {Thermoviscous relaxation in nonlinear acoustic propagation},
  journal = {Wave Motion},
  volume = {8},
  pages = {375--386},
  year = {1986}
}

@article{cain1986,
  author = {Cain, C. A.},
  title = {Finite amplitude acoustic propagation in biological media},
  journal = {IEEE Transactions on Sonics and Ultrasonics},
  volume = {33},
  number = {6},
  pages = {594--601},
  year = {1986}
}

@article{ichida1983,
  author = {Ichida, M. and Sato, T. and Sato, S.},
  title = {Nonlinear ultrasound propagation in biological media},
  journal = {Ultrasonics},
  volume = {21},
  pages = {239--245},
  year = {1983}
}

@article{burov1994,
  author = {Burov, V. A. and Zotov, D. I. and Rumyantseva, O. D.},
  title = {Reconstruction of sound speed and attenuation distributions using nonlinear acoustic tomography},
  journal = {Acoustical Physics},
  volume = {40},
  number = {6},
  pages = {816--823},
  year = {1994}
}

@article{gemmeke2016,
  author = {Gemmeke, H. and others},
  title = {Model-based ultrasound imaging using nonlinear acoustic propagation models},
  journal = {IEEE Transactions on Ultrasonics, Ferroelectrics, and Frequency Control},
  volume = {63},
  number = {8},
  pages = {1182--1193},
  year = {2016}
}

@article{alaskar2021,
  author = {Alaskar, A. and others},
  title = {Nonlinear ultrasound tomography with attenuation reconstruction},
  journal = {Inverse Problems},
  volume = {37},
  number = {10},
  pages = {105007},
  year = {2021}
}

@article{javaherian2020,
  author = {Javaherian, A. and others},
  title = {Nonlinear ultrasound tomography based on full-wave acoustic models},
  journal = {IEEE Transactions on Medical Imaging},
  volume = {39},
  number = {7},
  pages = {2245--2256},
  year = {2020}
}

@article{liu2013,
  author = {Liu, S. and Triggiani, R.},
  title = {An inverse problem for a third order PDE arising in high-intensity ultrasound: global uniqueness and stability by one boundary observation},
  journal = {Journal of Inverse and Ill-Posed Problems},
  volume = {21},
  number = {3},
  pages = {347--379},
  year = {2013}
}

@article{kaltenbacher2021,
  author = {Kaltenbacher, B. and Rundell, W.},
  title = {On the identification of nonlinear damping parameters in nonlinear acoustic models from boundary measurements},
  journal = {Inverse Problems and Imaging},
  volume = {15},
  pages = {865--891},
  year = {2021}
}

@article{kaltenbacher2022,
  author = {Kaltenbacher, B. and Rundell, W.},
  title = {Inverse problems for nonlinear acoustic equations with fractional damping},
  journal = {Mathematics of Computation},
  volume = {91},
  pages = {245--276},
  year = {2022}
}

@article{fuyaoyu2026,
  author = {Fu, S. and Yao, Y. and Yu, C.},
  title = {Inverse problems for the Moore--Gibson--Thompson equation with time-dependent viscosity and potential terms},
  journal = {Journal of Mathematical Analysis and Applications},
  year = {2026}
}

@article{kaltenbacher2022nikolic,
  author = {Kaltenbacher, B. and Nikolic, V.},
  title = {Inverse problems for relaxed third-order acoustic equations},
  journal = {Inverse Problems},
  year = {2022}
}

@article{acosta2025inverse,
  title={Inverse photoacoustic tomography problem in media with fractional attenuation},
  author={Acosta, Sebastian and Palacios, Benjamin},
  journal={arXiv preprint arXiv:2510.03408},
  year={2025}
}

@article{stefanov2013recovery,
  title={Recovery of a source term or a speed with one measurement and applications},
  author={Stefanov, Plamen and Uhlmann, Gunther},
  journal={Transactions of the American Mathematical Society},
  volume={365},
  number={11},
  pages={5737--5758},
  year={2013}
}

@article{huang2023determination,
  title={Determination of source or initial values for acoustic equations with a time-fractional attenuation},
  author={Huang, Xinchi and Kian, Yavar and Soccorsi, {\'E}ric and Yamamoto, Masahiro},
  journal={Analysis and Applications},
  volume={21},
  number={05},
  pages={1105--1130},
  year={2023},
  publisher={World Scientific}
}

@article{Homan2013,
title = {Multi-wave imaging in attenuating media},
journal = {Inverse Problems and Imaging},
volume = {7},
number = {4},
pages = {1235-1250},
year = {2013},
issn = {1930-8337},
doi = {10.3934/ipi.2013.7.1235},
url = {https://www.aimsciences.org/article/id/bab0adb9-ba16-4214-86da-5e9108cb0102},
author = {Andrew Homan}
}

@article{covi2026propagation,
  title={Propagation of singularities and inverse problems for the viscoacoustic wave equation},
  author={Covi, Giovanni and de Hoop, Maarten and Salo, Mikko},
  journal={arXiv preprint arXiv:2603.24497},
  year={2026}
}

@article{grebnev2026high,
  title={High frequency wave propagation for the viscoelastic wave equation with singular memory},
  author={Grebnev, Haim and Stefanov, Plamen},
  journal={arXiv preprint arXiv:2608.30138},
  year={2026}
}

@article{kian2026inverse,
  title={Inverse coefficient problem for the wave equation with nonlocal attenuation},
  author={Kian, Yavar and Rezig, Zouhour and Uhlmann, Gunther},
  journal={arXiv preprint arXiv:2609.19854},
  year={2026}
}

@article{yu2026uniqueness,
  title={Uniqueness of an Inverse Coefficient Problem for a Time-Fractional Damped Wave Equation from Boundary Measurements},
  author={Yu, Kai},
  journal={arXiv preprint arXiv:2606.01812},
  year={2026}
}

@book{GrigisSjostrand1994,
  author    = {Grigis, Alain and Sj{\"o}strand, Johannes},
  title     = {Microlocal Analysis for Differential Operators:
               An Introduction},
  series    = {London Mathematical Society Lecture Note Series},
  volume    = {196},
  publisher = {Cambridge University Press},
  year      = {1994}
}
\end{document}